\documentclass[11pt,a4paper,twoside]{article}

\usepackage{microtype}
\usepackage[T1]{fontenc}
\usepackage{helvet}

\usepackage{sfmath}

\usepackage[margin=2cm, top=4.5cm, bottom=3cm, head=3cm]{geometry}
\usepackage{lastpage}

\usepackage{array}
\usepackage[british]{babel}
\usepackage[squaren]{SIunits}
\usepackage{natbib}
\usepackage{url}
\usepackage{hyperref}
\usepackage{verbatim}
\usepackage{multirow}
\usepackage{subfigure}
\usepackage[ruled]{algorithm2e}
\usepackage[bf]{caption}
\usepackage{setspace}
\usepackage{float}
\usepackage{placeins}
\hypersetup{colorlinks=false}

\usepackage{tikz}
\usetikzlibrary{backgrounds}
\usetikzlibrary{calc}
\usetikzlibrary{shapes}

\usepackage{ulsy}
\usepackage[thmmarks]{ntheorem}
\usepackage{amsmath,amssymb}
\newtheorem{thm}{Theorem}
\newtheorem{lem}{Lemma}
\newtheorem{dfn}{Definition}
\newtheorem{rmk}{Remark}
\newtheorem{cor}{Corollary}
\newtheorem{prop}{Proposition}

\theoremsymbol{\ensuremath\blacksquare}
\theorembodyfont{\normalfont}
\newtheorem*{proof}{Proof}

\theoremsymbol{\blitza}

\usepackage{parskip}
\newcommand{\vect}[1]{\mathbf{#1}}

\usepackage{fancyhdr}
\pgfdeclareimage[height=1.5cm]{ucamlogo}{CUED_logo}

\usepackage{fancybox}
\title{\shadowbox{TECHNICAL REPORT: CUED/F-INFENG/TR.671}\\\vspace{\baselineskip}\bf Designing MPC controllers by reverse-engineering existing LTI controllers}
\author{E. N. Hartley\thanks{\tt enh20@eng.cam.ac.uk}{~}~and J. M. Maciejowski\thanks{\tt jmm@eng.cam.ac.uk}\\
~\\
\it(Cambridge University Engineering Department)
}
\begin{document}

\pagestyle{empty}
\begin{tikzpicture}[remember picture, overlay]

\path (current page.north west) -- ++(5.75,-9.5) coordinate (aaaa);

\draw[very thick, draw=blue!50!black] (aaaa) rectangle ++(9.5,-6) coordinate (bbbb);

\path (aaaa) -- node[pos=0.5,text width=8.5cm, text badly centered]{%
\fontfamily{pcr}%
\fontseries{b}
\fontsize{12}{15}
\selectfont%
{Designing MPC controllers by reverse-engineering existing LTI controllers}\\
\fontseries{m}%
\selectfont%
\vspace{\baselineskip}
E. N. Hartley, J. M. Maciejowski\\
\vspace{\baselineskip}
CUED/F-INFENG/TR.671\\
September 2011
} (bbbb);

\end{tikzpicture}

\cleardoublepage
\setcounter{page}{1}
\pagestyle{fancyplain}
\maketitle

\vfill
\begin{abstract}
This technical report presents a method for designing a constrained output-feedback model predictive controller (MPC) that behaves in the same way as an existing baseline stabilising linear time invariant output-feedback controller when constraints are inactive.  The baseline controller is cast into an observer-compensator form and an inverse-optimal cost function is used as the basis of the MPC controller.  The available degrees of design freedom are explored, and some guidelines provided for the selection of an appropriate observer-compensator realisation that will best allow exploitation of the constraint-handling and redundancy management capabilities of MPC.  Consideration is given to output setpoint tracking, and the method is demonstrated with three different multivariable plants of varying complexity.

\vspace{3\baselineskip}
\it\scriptsize
A paper based on the work presented in this technical report has been submitted to IEEE Transactions on Automatic Control on 12/09/2011 as a regular paper under the title ``Designing output-feedback predictive controllers by reverse-engineering existing LTI controllers''.
\end{abstract}
\vfill

\cleardoublepage
\tableofcontents 

\cleardoublepage
\section{Introduction}

In many cases, a conventional linear time-invariant (LTI) controller already exists for a given application where a system designer would consider the use of model predictive control (MPC) to improve performance.  The ability to handle input and output constraints in a systematic manner is one of the main reasons that motivates the use of MPC \citep{Mac2002,CB2004,RM2009} and the keystone of its industrial success \citep{QB2003}, allowing plants to safely operate more closely to the boundaries of their feasible operating regions.

Whilst the definition of the constraints is usually obvious for a given application, motivated by the physical limitations of the plant and performance requirements on the controlled variables, encoding the remaining control objectives into the cost function is often unintuitive, especially for a highly cross-coupled MIMO plant in the presence of unmeasured disturbances.

When full state measurements or estimates exist, an inverse-optimal control problem can be constructed to obtain a cost function for which the (unconstrained) optimum solution is equivalent to a prescribed state feedback gain \citep{Kalman1964}.  It was noted by \citep{KJ1972} that a quadratic cost function including cross-terms between state and input values could reproduce any multivariable state feedback gain, by making the primary control objective to reproduce the state feedback gain.  If a state feedback gain $K_c$ is within the domain of gains that can be obtained by solving the infinite horizon LQR problem, a linear matrix inequality (LMI) problem \citep{BGFB1994} can be posed to find quadratic cost weightings $Q \ge0$ and $R>0$ such that the infinite-horizon LQR controller is $K_c$.  In \citep{DB2009,DB2010} an LMI-based method allowing quadratic cost weights to vary throughout a finite prediction horizon is proposed allowing reproduction of a wider range of static gains when operating in the controller's unconstrained regime and that the behaviour of a dynamic feedback controller can also be reproduced by including the original controller dynamics within the plant model.

In reality, the full state measurements are not always directly available, necessitating further work for the system designer: the design of a state observer to obtain an estimate of the state.  The observer introduces additional dynamics that change the closed-loop behaviour.  Moreover, it is known that a ``fast'' observer and a ``good'' feedback gain do not necessarily combine to give good closed-loop performance \citep{Doyle1978,DS1979}, so this process is nontrivial.

To capitalise on the effort already expended during the design of the existing unconstrained LTI output-feedback controller, it would be desirable to be able to construct an initial MPC design and state observer, which, when used in combination, replicate the unconstrained behaviour of the original controller.

In \citep{RM1999,RM2000} a method was proposed for obtaining an MPC controller with the same unconstrained behaviour as an $H_{\infty}$ controller obtained through the \emph{loop-shaping} procedure of \citep{MG1992}, thus inheriting its desirable performance and robustness properties when constraints are inactive.  This paper instead describes a method in which an observer-compensator realisation of an arbitrary stabilising LTI output-feedback controller can be used to obtain the state observer, and cost function for the MPC controller.

The design method presented here relies upon the results of \citep{BF1985,Ben1985,FBA1986}, further developed in \citep{AA1999,DAC2006}, which show an analytical method for obtaining an observer-based realisation of an arbitrary linear time-invariant, stabilising, output feedback controller, and the methodology is motivated by the cross standard form (CSF), introduced in \citep{Ala2002a}. The CSF is an inverse-optimal generalised plant model with the optimal $\mathcal{H}_2$ and $\mathcal{H}_{\infty}$ controllers both equal to the observer-based realisations of a pre-specified output feedback controller $K_0$ of order greater than or equal to the order of the plant.  A discrete-time variant of the CSF of \citep{Ala2002a} is defined in \citep{VAAMC2003}, whilst in \citep{DAC2006} the continuous-time CSF is generalised to accommodate the case where the baseline controller is of lower order than the controlled plant.  

It was proposed in \citep{Mac2007} that a similar method can also be used as the starting point for a model predictive controller that, by construction, will exhibit the same behaviour as the original feedback controller in the absence of active constraints.  This paper builds upon the work of \citep{Mac2007,HM2009,JM2009, Hartley2010} by further addressing the effects of the choices of the non-unique realisation of the original controller in the context of the resulting constrained predictive controller.  Furthermore, a method is proposed, with which reference-tracking controllers can also be reverse-engineered and cast into the MPC framework.  The introduction of unwanted cross-coupling between seemingly unrelated control loops is also explained, and methods for avoiding this are proposed.

Three case studies are presented demonstrating the effectiveness of the procedure:
\begin{enumerate}
\item \hyperref[sec:scattitude]{Attitude control of a satellite with redundant torque pairs};
\item \hyperref[sec:invpend]{Control of an inverted pendulum}; and
\item \hyperref[sec:b747]{Control of a large airliner}.
\end{enumerate}

\section{Observer-based controller realisations}
\label{sec:obsbasedreworked}
MPC controllers are implemented on digital computers and usually operate in discrete time.  Also, the original controller could be of lower order than the plant model (for example, a proportional-integral controller), so a discrete-time variant of the method of \citep{DAC2006} for low order controllers is used to obtain the observer-compensator realisation.  The observer will be implemented directly, and the state feedback gain $K_c$ used as the basis for an MPC controller to which constraint handling will be added (Figure~\ref{fig:revengprocedure}).

\begin{figure*}[htbp]
\centering
\footnotesize
\begin{tikzpicture}
\tikzstyle{bdblock} = [draw, very thick, rounded corners, minimum height=0.75cm, minimum width=0.75cm]
\tikzstyle{bdarrow} = [draw, very thick, ->];

\begin{scope}
\node[bdblock](GZ){$G(z)$};
\node[bdblock, below of=GZ, node distance=1cm](KZ){$K(z)$};
\draw[bdarrow] (GZ) -- node[above]{$y$} +(1.25,0) |- (KZ);
\draw[bdarrow] (KZ) -- node[above]{$u$} +(-1.25,0) |- (GZ);

\node at (0,-2) (n1) {(a) Original system};
\end{scope}

\begin{scope}[xshift=4.75cm,yshift=0.5cm]
\node[bdblock](GZ){$G(z)$};
\path(GZ) -- +(-1,-1) coordinate (QQ);
\path(GZ) -- +(1,-1) coordinate (RR);

\node[bdblock](KC) at (QQ){$K_c$};
\node[bdblock](OBS) at (RR){$G_{\text{obs}}(z)$};

\draw[bdarrow] (GZ) -- node[above]{$y$} +(2.25,0) |- (OBS);
\draw[bdarrow] (OBS) -- node[above]{$\hat{x}$}(KC);

\draw[very thick](KC) -- node[above]{$u$} +(-1.5,0) coordinate (SS);
\draw[very thick] (SS) |- (GZ);
\draw[bdarrow] (SS) -- +(0,-1) -| (OBS);
\node at (0,-2.5) (n2) {(b) Observer based realisation};
\end{scope}

\begin{scope}[xshift=10.5cm,yshift=0.5cm]
\node[bdblock](GZ){$G(z)$};
\path(GZ) -- +(-1,-1) coordinate (QQ);
\path(GZ) -- +(1,-1) coordinate (RR);

\node[bdblock](KC) at (QQ){MPC};
\node[bdblock](OBS) at (RR){$G_{\text{obs}}(z)$};

\draw[bdarrow] (GZ) -- node[above]{$y$} +(2.25,0) |- (OBS);
\draw[bdarrow] (OBS) -- node[above]{$\hat{x}$}(KC);

\draw[very thick](KC) -- node[above]{$u$} +(-1.5,0) coordinate (SS);
\draw[very thick] (SS) |- (GZ);
\draw[bdarrow] (SS) -- +(0,-1) -| (OBS);
\node at (0,-2.5) (n3) {(c) MPC controller};
\end{scope}

\node at ($(n1.east)!0.5!(n2.west)$) {$\rightarrow$};
\node at ($(n2.east)!0.5!(n3.west)$) {$\rightarrow$};
\end{tikzpicture}
\caption{Schematic representation of the reverse-engineering procedure}
\label{fig:revengprocedure}
\end{figure*}
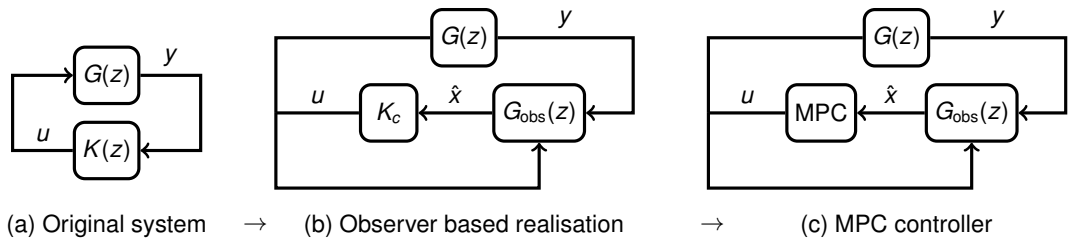

The principles for obtaining the observer-based realisation are now presented in sufficient detail to motivate the decisions made to obtain a satisfactory constrained controller.  Refer to the appendix for proofs.  Consider a linear discrete-time state-space plant model $G(z)$ of order $n$ with $n_y$ outputs and $n_u$ inputs, with pair $(C,A)$ observable, and an existing stabilising LTI controller $K_0(z)$ expressed as its minimal realisation with order $n_K \le n$,
\begin{equation*}
G(z) =
\left[
\begin{array}{c|c}
A & B\\
\hline
C & 0
\end{array}
\right],
\quad
K_0(z) =
\left[
\begin{array}{c|c}
A_K & B_K\\ \hline
C_K & D_K
\end{array}
\right].
\end{equation*}
There are two main structures for discrete-time observers of the plant $G(z)$ \citep{AA1999,Tei2008}.  These differ in whether the current measurements affect the current state estimate or not.

\begin{dfn}[Filter form observer]  A \emph{filter} structure discrete time observer provides an \emph{a posteriori} filtered estimate of the current plant state, and takes the form
\begin{align*}
\hat{x}(k+1|k) & = (A-AK_fC)\hat{x}(k|k-1)+Bu(k) + AK_fy(k)\\
\hat{x}(k|k) & = (I-K_fC)\hat{x}(k|k-1)+K_fy(k)
\end{align*}
where $K_f$ is an appropriately chosen observer gain matrix and $\hat{x}(k|k)$ is used to compute $u(k)$.
\end{dfn}

\begin{rmk}
Measurements from the current time step contribute to the estimate of the current state.  This is useful if the sampling period is significantly longer than the time required to compute the control action (i.e. can be modelled as a direct feed-through).
\end{rmk}

\begin{dfn}[Predictor-form observer]
The \emph{predictor} structure provides an \emph{a priori} prediction of the plant state by using the output measurements from the previous time step and takes the form
\begin{equation*}
\hat{x}(k+1|k) = (A-K_fC)\hat{x}(k|k-1) + Bu(k) + K_fy(k)
\end{equation*}
where $K_f$ is an appropriately chosen observer gain matrix such that $(A-K_fC)$ is stable.
\end{dfn}

\begin{rmk}
The state estimate $\hat{x}(k|k-1)$ is used for control purposes ---i.e. to compute $u(k)$--- at each time step.  By using the estimate of $x$ at time $k$ given measurements from time $k-1$ as the boundary condition for the beginning of the optimised trajectory, the optimisation associated with the MPC controller can commence before time $k$, allowing a time period equal to the sampling period for the computation to be performed.
\end{rmk}
\begin{rmk}
A controller with a non-zero direct feedthrough term $D_{K}$ (i.e. a controller that is not strictly proper) cannot be directly reproduced using the estimate from a discrete-time predictor.  Techniques to avoid this limitation will be described in Section~\ref{subsec:loopshift}.  In this section, assume that some transformation on (or modification to) the system has already been performed to ensure that the controller is strictly proper, yielding
\begin{equation*}
\tilde{G}(z) =
\left[
\begin{array}{c|c}
\tilde{A} & \tilde{B}\\
\hline
\tilde{C} & 0
\end{array}
\right]\,\quad
\tilde{K}_0(z) =
\left[
\begin{array}{c|c}
\tilde{A}_K & \tilde{B}_K\\
\hline
\tilde{C}_K & 0
\end{array}
\right].
\end{equation*}
\end{rmk}

\begin{thm}
\label{thm:zerosatorigin}
For a filter-form observer-based controller, $K_{\mathrm{obs}}(z) = \left[ \begin{array}{c|c}A_o&B_o\\\hline C_o&D_o\end{array}\right]$, $D_o = C_o A^{-1}_o B_o$, which implies that $K_{\mathrm{obs}}(0) = 0$.
\end{thm}
\begin{proof}
Consider the following observer-based controller:
\begin{equation}
K_{\mathrm{obs}}(z) =
\left[
\begin{array}{c|c}
(A+BK_c)(I-K_fC) & (A+BK_c)K_f\\
\hline
K_c(I-K_fC) & K_cK_f
\end{array}
\right].
\label{eqn:filterformk}
\end{equation}
Noting that $A_o^{-1} = (I-K_fC)^{-1}(A+BK_c)^{-1}$ it is clear that $D_o = C_o A^{-1}_o B_o$, which implies that $D_o + C_o\left(zI-A_o\right)^{-1}B_o = 0$, when $z=0$.
\end{proof}

\begin{rmk}
For
%
the observer-based realisation
to be a realisation of a controller $K_0(z)$, $K_0(z)$ must have this same property.  Therefore through term $D_K$ can only be correctly reproduced in discrete-time filter observer-based form if $D_K = C_KA_K^{-1}B_K$
\citep{BF1985}.
\end{rmk}

Given a matrix $T$ with full row rank, let the notation $T^{\dagger}$ indicate a right inverse of $T$, $T^{+}$ indicate the Moore-Penrose pseudo-inverse of $T$, and $T^{\bot}$ indicate a matrix whose columns form an orthonormal basis for the nullspace of $T$.

\begin{lem}
\label{lem:kckf}
Given a matrix $K_c = D_KC + C_KT$ and a matrix $K_f$ such that $AK_f = T^{\dagger}B_K - BD_K$, and assuming $D_K = C_KA_K^{-1}B_K$, $T$ is of full row-rank and $-T(A+BD_KC)-TBC_KT+B_KC+A_KT = 0$, then it holds that $K_cK_f = D_K$.
\end{lem}

\begin{proof}
\begin{align}
K_cK_f & = (D_KC+C_KT)A^{-1}(T^{\dagger}B_K-BD_K) \nonumber\\
& = (C_KA_K^{-1}B_KC + C_KT)A^{-1}(T^{\dagger}B_K-BC_KA_K^{-1}B_K)\nonumber\\
& = C_K(A_K^{-1}B_KC+T)A^{-1}(T^{\dagger}-BC_KA_K^{-1})B_K\nonumber\\
& = C_KA_K^{-1}(B_KC+A_KT)A^{-1}(T^{\dagger}-BC_KA_K^{-1})B_K.
\end{align}
By rearrangement and factorisation,
\begin{align}
TA & = -TBD_KC - TBC_KT + B_KC+A_KT \nonumber\\
 & = (I-TBC_KA_K^{-1})(B_KC+A_KT) \nonumber\\
\therefore \quad B_KC+A_KT & = (I-TBC_KA_K^{-1})^{-1}TA.
\end{align}
Therefore:
\begin{align}
K_cK_f & =C_KA_K^{-1}(I-TBC_KA_K^{-1})^{-1}TAA^{-1}(T^{\dagger}-BC_KA_K^{-1})B_K \nonumber\\
& = C_KA_K^{-1}(I-TBC_KA_K^{-1})^{-1}(I-TBC_KA_K^{-1})B_K \nonumber\\
& = C_KA_K^{-1}B_K\nonumber\\
& = D_K.
\end{align}
\end{proof}

\begin{thm}
\label{thm:filter}
Let $K_0(z)$ be a stabilising linear-time-invariant controller for the plant $G(z)$ where $n_K \le n$ and $K_0(0) = 0$.  Assume that there exists $T \in \mathbb{R}^{n_K \times n}$ of full row rank satisfying the non-symmetric Riccati equation
\begin{equation}
-T(A+BD_KC)-TBC_KT+B_KC+A_KT = 0
\label{eqn:filtermatch_ric}
\end{equation}
and that $\mathrm{det}(A) \ne 0$ and $\mathrm{det}(A_K) \ne 0$, then $K_0(z)$ can be realised in a filter observer-based form with observer gain $K_f$ such that $AK_f = T^{\dagger}B_K-BD_K$, and state-estimate feedback gain $K_c = D_KC+C_KT$.
\end{thm}

\begin{proof}
Consider the (non-minimal) realisation,
\begin{equation}
{K}_0(z) =
\left[
\begin{array}{cc|c}
A_K & 0 & B_K\\
A_{EK} & A_{E} & B_E\\
\hline
C_K & 0 & D_K
\end{array}
\right]
\label{eqn:augfilt}
\end{equation}
where $A_{EK} \in \mathbb{R}^{(n-n_K) \times n_K}$ and $A_{E} \in \mathbb{R}^{(n-n_K)\times(n-n_K)}$ and $B_{E} \in \mathbb{R}^{(n-n_K) \times n_y}$ have arbitrary values.
Given that $T$ is of full row rank, the matrix $\begin{bmatrix}T^T & T^{\bot}\end{bmatrix}^T$ is invertible, with inverse $\begin{bmatrix}T^{+} & T^{\bot}\end{bmatrix}$.  Consider the change of co-ordinates of (\ref{eqn:filterformk}) which gives
\begin{equation}
K_{\mathrm{obs}}(z) =
\left[
\begin{array}{c|c}
A_{\mathrm{obs}} & B_{\mathrm{obs}}\\
\hline
C_{\mathrm{obs}} & D_{\mathrm{obs}}
\end{array}
\right]
\label{eqn:augobsmatfiilter}
\end{equation}
where
\begin{subequations}
\begin{equation}
A_{\mathrm{obs}} =
\begin{bmatrix}
TM_1T^+ & TM_1T^{\bot}\\
T^{\bot^T}M_1T^+ & T^{\bot^T}M_1T^{\bot}
\end{bmatrix}
\end{equation}
with 
\begin{align}
M_1 
& = A - AK_fC + BK_c - BK_cK_fC
\end{align}
and
\begin{align}
B_\mathrm{obs} & = \begin{bmatrix}TAK_f + TBK_cK_f\\ T^{\bot^T}AK_f + T^{\bot^T}BK_cK_f\end{bmatrix}\\
C_\mathrm{obs} & = \begin{bmatrix}(K_c - K_cK_fC)T^+ & (K_c - K_cK_fC)T^{\bot}\end{bmatrix}\\
D_\mathrm{obs} & = K_cK_f = D_K.
\end{align}
Let $K_c = D_KC+C_KT$ and $AK_f = T^{\dagger}B_K-BD_K$ then
\begin{equation}
TM_1 = TA + TBD_KC - B_KC + TBC_KT
\end{equation}
\begin{equation}
B_\mathrm{obs}  =
\begin{bmatrix}
B_K\\
T^{\bot^T}T^{\dagger}B_K
\end{bmatrix}
\end{equation}
and
\begin{align}
C_{\mathrm{obs}} & =
\begin{bmatrix}
(K_c - D_KC)T^+ & (K_c - D_KC)T^{\bot}
\end{bmatrix} \nonumber\\
& = 
\begin{bmatrix}
C_K & 0
\end{bmatrix}.
\end{align}
\end{subequations}
If (\ref{eqn:filtermatch_ric}) holds then $TM_1 = A_KT$ and $TM_1T^+ = A_K$, and $TM_1T^{\bot} = 0$.
System  (\ref{eqn:filterformk}) is related to (\ref{eqn:augobsmatfiilter}) by similarity transformation, and (\ref{eqn:augobsmatfiilter}) is equal to
(\ref{eqn:augfilt}) with
\begin{subequations}
\begin{align}
A_{EK} & = T^{\bot^T}(A-T^{\dagger}B_KC +BD_KC + BD_KC + BC_KT - BD_KC)T^+\nonumber\\
& = T^{\bot^T}(A+BD_KC - T^{\dagger}BC_K)T^+ + T^{\bot^T}BC_K\\
A_{E} & = T^{\bot^T}(A - AK_fC + BK_c - BK_cK_fC)T^{\bot}\nonumber\\
& = T^{\bot^T}(A - T^{\dagger}B_KC - BD_KC + BC_KT + BD_KC - BD_KC)T^{\bot}\nonumber\\
& = T^{\bot^T}(A+BD_KC - T^{\dagger}B_KC)T^{\bot}\\
B_{E} & =T^{\bot^T}T^{\dagger}B_K
\end{align}
\end{subequations}
which in turn is a (non-minimal) realisation of $K_0(z)$. Therefore (\ref{eqn:filterformk}) is a realisation of $K_0(z)$.
\end{proof}

\begin{thm}
\label{thm:predictor}
Let $\tilde{K}_0(z)$ be a strictly proper, stabilising linear-time-invariant controller  for the plant $\tilde{G}(z)$ where $n_K \le n$.  Assume  that there exists $T \in \mathbb{R}^{n_K \times n}$ of full row rank satisfying the non-symmetric Riccati equation
\begin{equation}
-T\tilde{A} - T \tilde{B}\tilde{C}_KT + \tilde{B}_K\tilde{C}
+ \tilde{A}_KT = 0.
\label{eqn:predictorform_ric}
\end{equation}
Then $\tilde{K}_0(z)$ can be realised in a discrete-time predictor observer-based form with the observer gain $K_f = T^{\dagger}\tilde{B}_K$ and the state-estimate feedback gain $K_c = \tilde{C}_KT$.
\end{thm}
\begin{proof}
The unconstrained observer-based controller is of the form:
\begin{equation}
\tilde{K}_{\mathrm{obs}}(z) =
\left[
\begin{array}{c|c}
\tilde{A} - K_f \tilde{C} + \tilde{B}K_c & K_f\\
\hline
K_c & 0
\end{array}
\right].
\label{eqn:obsbased1}
\end{equation}
Consider a non-minimal realisation of $\tilde{K}_0(z)$ with state vector $\begin{bmatrix}x_K^T & x_E^T\end{bmatrix}^T$ of the form
\begin{equation}
\tilde{K}_0(z) =
\left[
\begin{array}{cc|c}
A_K & 0 & B_K\\
A_{EK} & A_{E} & B_E\\
\hline
C_K & 0 & 0
\end{array}
\right]
\label{eqn:augpred}
\end{equation}
where $A_{EK} \in \mathbb{R}^{(n-n_K) \times n_K}$, $A_{E} \in \mathbb{R}^{(n-n_K)\times(n-n_K)}$ and $B_{E} \in \mathbb{R}^{(n-n_K) \times n_y}$.
Given that $T$ is of full row rank, the matrix $\begin{bmatrix}T^T & T^{\bot}\end{bmatrix}^T$ is invertible, with inverse $\begin{bmatrix}T^{+} & T^{\bot}\end{bmatrix}$.  Consider the similarity transformation on (\ref{eqn:obsbased1}) which gives
\begin{equation}
\tilde{K}_{\mathrm{obs}}(z) =
\left[
\begin{array}{c|c}
A_{\mathrm{obs}} & B_{\mathrm{obs}}\\
\hline
C_{\mathrm{obs}} & 0
\end{array}
\right]
\label{eqn:augobsmat}
\end{equation}
where
\begin{subequations}
\begin{equation}
A_{\mathrm{obs}}  =
\begin{bmatrix}
TM_1T^+ & TM_1T^{\bot}\\
T^{\bot^T}M_1T^+ & T^{\bot^T}M_1T^{\bot}
\end{bmatrix}
\end{equation}
with $M_1 = (\tilde{A} - K_f\tilde{C} + \tilde{B}K_c)$,
\begin{align}
B_{\mathrm{obs}} & = 
\begin{bmatrix}
T K_f\\
T^{\bot^T}K_f
\end{bmatrix}\\
C_{\mathrm{obs}} & =
\begin{bmatrix}
K_c T^+ & K_c T^{\bot}
\end{bmatrix}.
\end{align}
\end{subequations}
By substituting $K_c = C_KT$ and $K_f = T^{\dagger}B_K$,
\begin{subequations}
\begin{equation}
A_{\mathrm{obs}} =
\begin{bmatrix}
M_2T^+ + T\tilde{B}C_K & M_2T^{\bot}\\
T^{\bot^T}M_3T^+ + \tilde{B}\tilde{C}_K & T^{\bot^T}M_3T^{\bot}
\end{bmatrix}
\end{equation}
with $M_2 = (T\tilde{A} - \tilde{B}_K\tilde{C})$, and $M_3 = (\tilde{A} - T^{\dagger}\tilde{B}_K\tilde{C})$,
\begin{align}
B_{\mathrm{obs}} & = 
\begin{bmatrix}
B_K\\
T^{\bot^T}T^{\dagger}B_K
\end{bmatrix}\\
C_{\mathrm{obs}} & =
\begin{bmatrix}
C_K & 0
\end{bmatrix}.
\end{align}
\end{subequations}
If $T\tilde{A} + T\tilde{B}\tilde{C}_KT - \tilde{B}_K\tilde{C} = A_KT$ then it follows that
\begin{equation}
(T\tilde{A} - \tilde{B}_K\tilde{C})T^+ + T\tilde{B}C_K = A_K.
\end{equation}
It also follows that
\begin{align}
(T\tilde{A} - \tilde{B}_K\tilde{C})T^{\bot} & = (A_KT - T\tilde{B}\tilde{C}_KT)T^{\bot} \nonumber\\
& = 0.
\end{align}
The observer-based controller (\ref{eqn:obsbased1}) with the specified $K_c$, $K_f$ and conditions on $T$ is related by similarity transformation to a system that is identical to (\ref{eqn:augpred}) when:
\begin{subequations}
\begin{align}
B_{E} & = T^{\bot^T}T^{\dagger}B_K\\
A_{EK} & = T^{\bot^T}(\tilde{A} - T^{\dagger}\tilde{B}_K\tilde{C})T^+ + \tilde{B}\tilde{C}_K\\
A_{E} & = T^{\bot^T}(\tilde{A} - T^{\dagger}\tilde{B}_K\tilde{C})T^{\bot}.
\end{align}
\end{subequations}
\end{proof}

\begin{rmk}
Equations (\ref{eqn:filtermatch_ric}) and (\ref{eqn:predictorform_ric}) can more conveniently be written in matrix form
\begin{equation}
\begin{bmatrix}
-T & I
\end{bmatrix}
A_{\mathrm{cl}}
\begin{bmatrix}
I\\T
\end{bmatrix}
= 0
\label{eqn:genric}
\end{equation}
where $A_{\mathrm{cl}}$ is the state update matrix for the original closed loop system,
\begin{equation}
A_{\mathrm{cl}} =
\begin{bmatrix}
A+BD_KC & BC_K\\
B_KC & A_K
\end{bmatrix}
\text{ or }
\begin{bmatrix}
\tilde{A} & \tilde{B}\tilde{C}_K\\
\tilde{B}_K\tilde{C} & \tilde{A}_K
\end{bmatrix}
.
\end{equation}
\end{rmk}
When $n_K < n$, the observer-based controller realisation has $n-n_K$ more closed loop modes than the original controller.  These correspond to the eigenvalues, with corresponding eigenvectors in the null-space of $T$. Whilst the dynamics of the output of the observer-based controller do not depend on these,  once $K_c = C_KT+D_KC$ is replaced by a constrained MPC controller, this will no longer be the case.  The value of $K_f$ is non-unique because $T^{\dagger}$ is non-unique.  Matrix $A_{E}$ is also non-unique for both of the described formulations.  Theorem~\ref{thm:freepolesassign} shows how these extra poles that are introduced into the closed loop system are determined by $T^{\dagger}$ and that they may be used to tune closed-loop constrained performance when constrained MPC is used instead of static gain $K_{c}$.

\begin{thm}
\label{thm:freepolesassign}
The $n-n_K$ additional modes in the observer error dynamics can be determined by the choice of a predictor-form observer gain for the plant
\begin{equation}
\left[
\begin{array}{c|c}
A+BD_KC & T^{\bot^T}{B}\\
\hline
{B}_K{C}T^{\bot} & 0
\end{array}
\right] 
\text{ or }
\left[
\begin{array}{c|c}
\tilde{A} & T^{\bot^T}\tilde{B}\\
\hline
\tilde{B}_K\tilde{C}T^{\bot} & 0
\end{array}
\right].
\label{eqn:freepolesystem}
\end{equation}
\end{thm}

\begin{proof}
The $n-n_K$ additional modes in observer error dynamics are associated with an invariant subspace of $A+BD_KC-AK_fC$ in the nullspace of $T$ and are determined by the eigenvalues of $A_E$.  Given that $AK_f = T^{\dagger}B_K - BD_K$ and that $T^{\dagger} = T^+ + T^{\bot}X$,
\begin{align}
A_E & = T^{\bot^T}\left(A - T^{\dagger}B_KC + BD_KC\right)T^{\bot} \nonumber\\
& = T^{\bot^T}\left(A - (T^{+}+T^{\bot}X){B}_K{C} + BD_KC\right)T^{\bot} \nonumber\\
& = T^{\bot^T}(A+BD_KC)T^{\bot} - XB_KCT^{\bot}.
\end{align}
The matrix $X$ is then treated as an observer gain, and can be designed by pole placement, or by Kalman filter methods on the system (\ref{eqn:freepolesystem}).
The proof for the predictor form is analogous.
\end{proof}

There are multiple solutions to the non-symmetric Riccati equations, which can be obtained using invariant subspace methods \citep{Lau1979}.
\begin{thm}
If the plant model $G(z)$ has $n$ states, and the original controller $K_0(z)$ has $n_K$ states, then (letting $\mathrm{Im}(\cdot)$ denote the image operator), given an $n$-dimensional invariant subspace, $\mathcal{S} \subset \mathbb{C}^{n+n_K}$ of $A_{cl}$,
\begin{equation}
\mathcal{S} \triangleq \mathrm{Im}
\begin{bmatrix}
|& & |\\
u_1 & \cdots & u_n\\
|&&|
\end{bmatrix}
\end{equation}
the columns can be partitioned vertically so that
\begin{equation}
\begin{bmatrix}
|& & |\\
u_1 & \cdots & u_n\\
|&&|
\end{bmatrix} = \begin{bmatrix}U_1\\U_2\end{bmatrix}
\label{eqn:U1U2}
\end{equation}
where $U_1 \in \mathbb{C}^{n\times n}$ and $U_2 \in \mathbb{C}^{n_K\times n}$.  $T = U_2U_1^{-1}$ is a solution to the non-symmetric Riccati equation (\ref{eqn:filtermatch_ric}) or (\ref{eqn:predictorform_ric}).
\end{thm}
\begin{proof}
Because $\begin{bmatrix}U_{1}^T &U_{2}^T\end{bmatrix}^T$ is an invariant subspace,
\begin{equation}
\begin{bmatrix}
A+BD_KC & BC_K\\
B_KC & A_K
\end{bmatrix}
\begin{bmatrix}
U_1\\
U_2
\end{bmatrix}
=
\begin{bmatrix}
U_1\\
U_2
\end{bmatrix}
\Lambda.
\end{equation}
Postmultiplying both sides by $U_1^{-1}$ yields
\begin{equation}
\begin{bmatrix}
A+BD_KC & BC_K\\
B_KC & A_K
\end{bmatrix}
\begin{bmatrix}
I\\
U_2U_1^{-1}
\end{bmatrix}
=
\begin{bmatrix}
I\\
U_2U_1^{-1}
\end{bmatrix}
U_1 \Lambda U_1^{-1}.
\label{eqn:u1eigenvector}
\end{equation}

By defining $T = U_{2}U_{1}^{-1}$ and premultiplying both sides by $\begin{bmatrix}-T & I\end{bmatrix}$, it follows that $T$ is the solution to the non-symmetric Riccati equation
\begin{equation}
\begin{bmatrix}
-T & I
\end{bmatrix}
\overbrace{
\begin{bmatrix}
A+BD_KC & BC_K\\
B_KC & A_K
\end{bmatrix}}^{A_{\mathrm{cl}}}
\begin{bmatrix}
I\\
T
\end{bmatrix}
=
0.
\end{equation}
\end{proof}

\begin{thm}
\label{thm:clpoles}
The poles of the pure state feedback system, $(A+BK_c)$ with $K_c$ calculated as in Theorem~\ref{thm:filter} or \ref{thm:predictor} are equal to the eigenvalues corresponding to the eigenvectors which comprise the invariant subspace $\mathcal{S} = \mathrm{Im}[U_1^T, U_2^T]^T$ if $T = U_2U_1^{-1}$.
\end{thm}

\begin{proof}
By considering (\ref{eqn:u1eigenvector}) it can be seen that
\begin{equation}
A + BD_KC + BC_KT = A+BK_c = U_{1}\Lambda U_{1}^{-1}.
\end{equation}
Therefore, $\sigma(A+BK_c) = \sigma(\Lambda)$.
\end{proof}
\begin{cor}
The remaining $n_K$ closed loop poles from the original system, along with the $n-n_K$ modes introduced if the observer is of higher order than the original controller therefore correspond to the observer error dynamics.
\end{cor}
\begin{rmk}
A real solution $T$ will not exist if the partition of closed-loop poles between pure state feedback and observer dynamics is not compatible with the controllability and observability properties of the original plant.  Also, complex conjugate pole pairs should not be split \citep{BF1985,AA1999}.
\end{rmk}
The resulting observer-based controller using $K_c$ and $K_f$ is a (non-minimal when $n_K < n$) realisation of the original controller.  The closed loop system using the observer based realisation of the controller will contain $n - n_K$ poles that did not exist in the original closed loop system.  These dynamics can be assigned by the system designer through the choice of $T^{\dagger}$ used to calculate $K_f$.

\section{Model predictive controller formulation}
At the heart of every MPC controller is a constrained optimisation problem, where the summation of a stage cost function of the plant input and state is optimised over a prediction horizon of length $N$, subject to input and state constraints \citep{Mac2002,CB2004,RM2009}.

Let $N$ be the length of the prediction horizon, $\ell(x,u)$ be the ``stage cost'' at each time step, and $F_{N}(x)$ be cost on the state at the end of the finite prediction horizon.  Define $\mathbb{X}$, $\mathbb{Y}$, $\mathbb{U}$ and $\mathbb{T}$ to be the set of feasible state values, the set of feasible output values, the set of feasible input values, and the terminal constraint set respectively.  A basic model predictive control formulation is outlined in Algorithm~\ref{alg:basicmpc}, using the shorthand notation $x(k)$ to be the prediction of the state $x$ at $k$ time steps into the future from the current state estimate or measurement, and $u(k)$  analogously.  For notational convenience, define:
\begin{align*}
\vect{x} & = 
\begin{bmatrix}
x(0)^T & \cdots & x(N)^T
\end{bmatrix}^T\\
\vect{u} & =
\begin{bmatrix}
u(0)^T & \cdots & u(N-1)^T
\end{bmatrix}^T.
\end{align*}
\begin{algorithm}[htbp]
\SetAlgoLined
\caption{Model predictive control}
\label{alg:basicmpc}
\While{controller running}{
\nl Sample state measurement or estimate $\hat{x}(t)$.

\nl Solve
\begin{equation*}
\arg \min_{\vect{x},\vect{u}}
F_N\left( x(N) \right) + \sum_{k=0}^{N-1} \ell\left( x(k), \, u(k) \right)
\end{equation*}
subject to plant dynamics
\begin{align*}
x(k+1) & = f\left( x(k), u(k) \right)\\
y(k+1) & = g\left( x(k), u(k) \right)
\end{align*}
and constraints
\begin{align*}
x(0) & = \hat{x}(t) \quad \text{(Current state measurement/estimate)}\\
x(k) & \in \mathbb{X} \quad \forall k \in \left\{0,\ldots,N-1 \right\}\\
y(k) & \in \mathbb{Y} \quad \forall k \in \left\{0,\ldots,N-1 \right\}\\
u(k) & \in \mathbb{U} \quad \forall k \in \left\{0,\ldots,N-1 \right\}\\
x(N) & \in \mathbb{T}.
\end{align*}
\nl Apply $u(0)$ to plant.

\nl Wait sampling time $T_s$.}
\end{algorithm}

For the case of a linear time invariant plant models, as used for the reverse-engineering:
\begin{align*}
f\left( x(k), u(k) \right) & = Ax(k) + Bu(k)\\
g\left( x(k), u(k) \right) & = Cx(k).
\end{align*}

\subsection{Zero-value cost-function}
The reverse engineering procedure hinges upon replacing the static gain $K_c$ with a constrained MPC controller that is, when constraints are not active, equivalent to the estimated state feedback $K_c \hat{x}$ obtained for the discrete-time observer-based controller.  A zero-value infinite horizon cost function can be constructed to ensure that $u(k) = K_c\hat{x}(k)$ is the optimal solution by using a stage cost
\begin{equation}
\ell(x(k),u(k)) =
\begin{bmatrix}
x(k)\\u(k)
\end{bmatrix}^T
\begin{bmatrix}
{K}_c^T R{K}_c & -{K}_c^TR\\
-R{K}_c & R
\end{bmatrix}
\begin{bmatrix}
x(k) \\ u(k)
\end{bmatrix}
\label{eqn:stagecost}
\end{equation}
where $R > 0$ is a weighting matrix, determining the relative importance of matching each input to that provided by the original controller \citep{KJ1972}.  A standard MPC implementation performs an optimisation over a finite, but receding horizon.
A finitely parameterised infinite horizon cost function can be obtained by using the candidate cost function over a finite horizon of length $N$ and using the solution $P$ to the discrete-time algebraic Riccati equation as a terminal quadratic cost weighting \citep{RM1993, CM1996}.
\begin{thm}
When using the stage cost (\ref{eqn:stagecost}), $P=0$ is a solution to the associated discrete time algebraic Riccati equation (DARE).
\end{thm}
\begin{proof}
The discrete time algebraic Riccati equation associated with stage cost (\ref{eqn:stagecost}) is
\begin{equation}
A^TPA - P - \left(A^TPB - K_c^TR \right)\left( B^TPB + R\right)^{-1}\left( B^TPA - RK_c \right)
 + K_c^TRK_c = 0.
\end{equation}
By construction, the optimal state feedback gain is, $K_c = -\left(B^TPB+R\right)^{-1}\left(B^TPA - RK_c\right)$, so
\begin{align}
A^TPA - P  + \left(A^TPB - K_c^TR \right)K_c + K_c^TRK_c & = 0 \nonumber\\
A^TPA - P + A^TPBK_c & = 0 \nonumber\\
A^TP(A+BK_c) - P & = 0. \label{eqn:ihzero}
\end{align}
By inspection, $P=0$ is a solution.
\end{proof}
\begin{rmk}
Because of this unsurprising result, no terminal cost need be added to the finite horizon optimisation.  
\end{rmk}
\begin{rmk}
When input constraints are active, the MPC implementation is interpreted as ``regulating'' to the region of the state space where $u(k)=K_cx(k)$ is feasible.  This is more intelligent than merely ``clipping'' the control actions on input saturation.  When output constraints are active, the control objective is to avoid constraint violations whilst minimising deviations from the unconstrained control actions over the horizon.
\end{rmk}
When there is input redundancy, if one control actuator fails or saturates, the controller should be capable of using other control inputs to achieve a similar control effect.  Under nominal operating conditions retaining the behaviour of the original controller is also desirable.  To meet these objectives, a possible quadratic cost function is
\begin{equation*}
\ell\left(x(k), u(k) \right) = \left\| Bu(k) - BK_cx(k) \right\|_{Q_1}^{2}
+ \left\| u(k) - K_cx(k) \right\|_{R_1}^{2}
\end{equation*}
for $R_1>0$, and $Q_1 \gg R$.  The first term, tries to achieve the control effect of the original controller, whilst the second term ensures that when feasible, the original control actuator configuration is used.  This cost function can alternatively be expressed in the form of (\ref{eqn:stagecost}) as:
\begin{equation}
\ell\left(x(k), u(k)\right) =
\begin{bmatrix}
K_c^T (R_1+B^T Q_1 B)K_c & -K_c^T(R_1+B^T Q_1B)\\
-(R_1+B^TQ_1B)K_c & (R_1 + B^TQ_1B)
\end{bmatrix}.
\label{eqn:stattrack}
\end{equation}
\begin{rmk}
Despite the resemblance, the predictive nature of the MPC implementation means that there is an anticipatory aspect to the control decision rather than a best-effort attempt to deliver a particular control effect at the current time step as would occur with a pure actuator allocation algorithm.
\end{rmk}

Alternative zero-value cost functions such as
\begin{subequations}
\begin{align}
\ell\left(x(k), u(k) \right) & = \|R(u-K_cx)\|_1\\
\text{ or } \ell\left(x(k), u(k) \right) & = \|R(u-K_cx)\|_{\infty}
\end{align}
\label{eqn:1normcosts}
\end{subequations}%
or a combination of the two, can also be used depending on the desired control objectives when the baseline controller does not satisfy constraints \citep{RR2000}.

\subsection{Other cost functions}
If there happens to be a valid $Q \ge0$ and $R\ge0$, with $S=0$ such that $K_c$ is the optimal infinite-horizon discrete-time state feedback minimising the DLQR cost function
\begin{equation}
\sum_{k=0}^{\infty} x(k)^T Q x(k) + u(k)^T R(u) + 2x(k)^T S u(k)
\end{equation}
 then LMI-based inverse optimality methods \citep{BGFB1994} can be used.  Even if this is not the case, the cross term $S$ can be minimised using LMI-based methods (Algorithm~\ref{alg:mincross}).  In the general case though, the cross-terms are not guaranteed to be driven to zero, however the resulting cost function can provide different constrained closed-loop behaviour.
\begin{algorithm}[h]
For a fixed $\epsilon > 0$, and compatibly sized $\overline{Q}$, $\overline{P}$, $\overline{R}$, $\overline{S}$ minimise $\|S\|_{2}^2$ subject to:
\begin{subequations}
\begin{align}
\overline{Q} &\ge 0\\
\overline{P} &\ge 0\\
\overline{R} &\ge \epsilon I\\
A^T\overline{P}A - \overline{P} - K_c^T(B^T\overline{P}B+R)K_c + \overline{Q} & =  0; \text{ and}\\
(B^T\overline{P}B+\overline{R})K_c +(B^T\overline{P}A+\overline{S}^T) & = 0.
\end{align}
\end{subequations}
\caption{Minimising cross terms using an LMI}
\label{alg:mincross}
\end{algorithm}

An approximate match to the original control gain $K_c$ can be obtained by using Algorithm~\ref{alg:mincross} and setting $\overline{S}=0$ \emph{a posteriori} and calculating:
\begin{equation}
\overline{K}_c = -\text{\tt dlqr}(A,B,\overline{Q},\overline{R}).
\end{equation}
The suitability of the newly synthesised controller for the given application is not guaranteed and would have be verified experimentally.

\begin{rmk}
This method for minimising $\overline{S}$ is not suitable when the cost function (\ref{eqn:stattrack}) has been used for a plant with redundant inputs.  The elements of $\overline{R}$ corresponding to redundant actuators that are not normally used will be forced towards infinity.
\end{rmk}

\begin{rmk}
If the plant model has been augmented with a disturbance model, the disturbance states are uncontrollable.  The cost-function re-shaping should be performed minimising only the cross-terms between inputs and the controllable states.  It is clear that the cross terms between the disturbance states and the input must remain, and in fact, this can be interpreted as an implicit input target calculator, as commonly used for offset-free MPC \citep{MB2002}.  There is, however, still no guarantee that the remaining terms of $\overline{S}$ will be forced to zero.
\end{rmk}

\begin{rmk}
The cost functions (\ref{eqn:stagecost}) and (\ref{eqn:stattrack}) have clear physical interpretations.  On the other hand the results of Algorithm~\ref{alg:mincross} are no so easily interpreted unless the elements of $\overline{S}$ are small in relation to the other matrices or the controller synthesised with $\overline{S}$ artificially set to zero is acceptable.
\end{rmk}

As an alternative, the method of \citep{DB2009,DB2010} can also be directly applied to increase the set of gains $K_c$ that can be matched without cross-terms between state and input, although it is still not guaranteed that a solution will exist for arbitrary $K_c$.

\begin{prop}
If inverse-optimal cost weightings $Q_1$ and $R_0$ cannot be found for $u = K_c x(0)$ to be the optimum solution of the one step horizon control problem  
\begin{equation}
\min_{u(0)} \left(Ax(0)+Bu(0)\right)^T Q_1 \left(Ax(0)+Bu(0)\right) + u_0^TR_0u_0
\label{eqn:nocosttogo}
\end{equation}
then there is no sequence of $(Q_i,R_i)$ over any finite prediction horizon $N$ which will give $u^*(0) = K_c x(0)$.
\end{prop}
\begin{proof}
Any finite-horizon optimal control problem with a quadratic cost function (with, or without cross terms between inputs and states and different time steps) will have a closed-form matrix quadratic expression of form $x^T P x$ for the optimum cost.  Therefore, if there is no suitable cost-to-go $Q_1$ for which the minimising control input of (\ref{eqn:nocosttogo}) is $K_c x(0)$ there is no suitable quadratic function of the prediction horizon from $k=1$ to $k=N$ (for arbitrary $N$) either.
\end{proof}

\subsection{Constrained stability}

In the presence of constraints, the (unconstrained) stabilising properties of the original controller might not be inherited.  The usual technique of applying a terminal constraint as in \citep{MRRS2000} can be applied directly if formal stability guarantees are required.

Consider a finite horizon MPC controller using a stage cost function (\ref{eqn:stagecost}), input constraints $u(k) \in \mathbb{U}$, state constraints $x(k) \in \mathbb{X}$, and a terminal constraint $x(N) \in \mathbb{T}$, where $\mathbb{T}$ is a recursively feasible positively invariant set under $K_c$ such that
\begin{subequations}
\begin{align}
x(k) \in \mathbb{T} & \implies K_cx(k) \in \mathbb{U}\\
x(k) \in \mathbb{T} & \implies x(k) \in \mathbb{X}\\
x(k) \in \mathbb{T} & \implies y(k) \in \mathbb{Y}\\
x(k) \in \mathbb{T} & \implies (A+BK_c)x(k) \in \mathbb{T}.
\label{eqn:invsetreveng}
\end{align}
\end{subequations}
Assuming exact plant model matching, and that the observer error has converged to zero, the reverse engineered MPC controller can be interpreted as a special form of \emph{dual-mode} MPC controller, with $x(k)$ guaranteed to enter $\mathbb{T}$ in $N$ steps.  Inside $\mathbb{T}$, by construction, the MPC controller is equivalent to the original LTI output-feedback controller, and therefore inherits its stabilising properties.

\section{Baseline controller transformations}

\subsection{Discretisation}

Real-world systems operate in continuous time, and an existing controller might be specified in continuous time, but practical implementations of MPC operate in discrete-time with sampled data.  In \cite{Mac2007} it was proposed to find a continuous-time observer-compensator based realisation of an original continuous-time controller, to implement the observer in continuous time and then sample the output.  If the gain used to form stage cost (\ref{eqn:stagecost}) is the same as the continuous time state feedback gain, the closed loop system behaviour might be very different to that when using the original controller.
Better output-performance matching can be achieved by finding an ``equivalent'' discrete-time cost weighting matrix \citep{Van1978} that minimises an integral cost function whilst being formulated as a discrete time problem  as done in \citep{Mac2007}.  This however, changes the effective gain of the unconstrained controller such that its rows are no longer in the row space of $T$.  As a consequence, the $n - n_K$  modes corresponding to error dynamics in the nullspace of $T$ will affect even the unconstrained closed loop system.  This added complication constitutes a strong argument for discretising the baseline controller and the plant model first, and directly obtaining the discrete-time observer-based realisation.

The usual zero-order hold method best models how a simple MPC controller would drive a real plant and this should be used for discretising the plant model.  For the controller, it might be preferable to use a first-order hold or a Tustin transformation (e.g. \citep{FPW1990}), particularly at low sampling frequencies.  These can introduce non-zero $D_K$ terms even if none existed in the continuous time controller.  This process should, therefore, be performed before any further transformations are performed to comply with the restrictions of the observer-based realisations.

\subsection{Ensuring a strictly proper controller model}
\label{subsec:loopshift}
When a predictor structure is used for the observer, $D_K$ must be zero.  Two options are available when $D_K$ is non-zero.

\subsubsection{Loop-shifting}
Loop-shifting \citep{ZDG1996} can be used (Figure~\ref{fig:loopshift}a), leading to the following modified plant and controller:
\begin{subequations}
\begin{equation}
\tilde{G}(z) = 
\left[\begin{array}{c|c}
A+BD_{K}C & B\\
\hline
C & 0
\end{array}\right]
\end{equation}
\begin{equation}
\tilde{K}_0(z) =
\left[
\begin{array}{c|c}
A_K & B_K\\
\hline
C_K & 0
\end{array}
\right].
\end{equation}
\end{subequations}
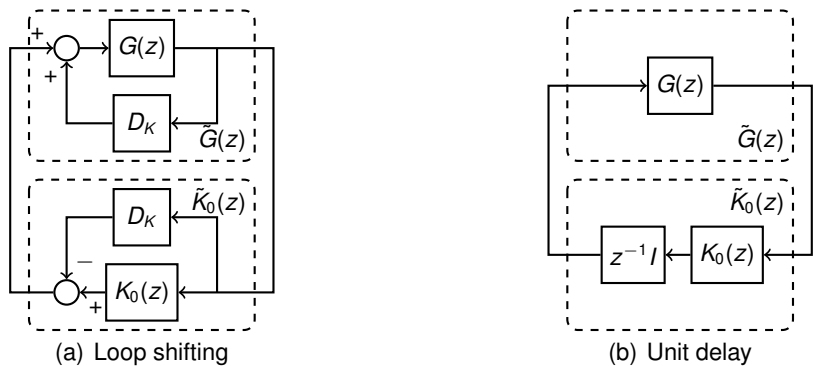
\begin{figure}[htb]
\centering
\footnotesize
\hfill
\subfigure[Loop shifting]{
\begin{tikzpicture}
\tikzstyle{bdblock} = [draw, minimum height=0.75cm, minimum width=0.75cm,  thick];

\tikzstyle{sumblock} = [draw, minimum height=0.325cm, minimum width=0.325cm,  thick, circle];

\tikzstyle{bdarrow} = [draw,  thick, ->];

\node[bdblock] (GZ) at (0,0) {$G(z)$};
\node[bdblock] (KZ) at (0,-3.25) {$K_0(z)$};

\node[bdblock] (DK1) at (0,-1) {$D_K$};
\node[bdblock] (DK2) at (0,-2.25) {$D_K$};

\node[sumblock] (S1) at (-1,0){};
\node[sumblock] (S2) at (-1,-3.25){};

\draw[bdarrow] (DK1) -| node[pos=1, below left]{$+$} (S1);
\draw[bdarrow] (S1) -- (GZ);

\path(KZ) -- +(1,0) coordinate (y2);
\draw[bdarrow] (GZ) -- ++(1,0) coordinate (y1) -- ++(0.75,0) |- (y2) -- (KZ);
\draw[bdarrow] (y1) |- (DK1);

\draw[bdarrow] (DK2) -| node[pos=1, above right]{$-$} (S2);
\draw[bdarrow] (KZ) -- node[pos=1, below right]{$+$} (S2);

\draw[bdarrow] (S2) -- ++(-0.75,0) |- node[pos=1, above left]{$+$} (S1);

\draw[bdarrow] (y2) |- (DK2);

\begin{pgfonlayer}{background}
\draw[thick, dashed, rounded corners](-1.5,0.5) rectangle (1.5,-1.5);
\draw[thick, dashed, rounded corners](-1.5,-1.75) rectangle (1.5,-3.75);

\node[above left] at (1.5,-1.5){$\tilde{G}(z)$};
\node[below left] at (1.5,-1.75){$\tilde{K}_0(z)$};
\end{pgfonlayer}
\end{tikzpicture}
}
\hfill
\subfigure[Unit delay]{
\begin{tikzpicture}
\tikzstyle{bdblock} = [draw, minimum height=0.75cm, minimum width=0.75cm,  thick];

\tikzstyle{sumblock} = [draw, minimum height=0.325cm, minimum width=0.325cm,  thick, circle];

\tikzstyle{bdarrow} = [draw,  thick, ->];

\node[bdblock] (GZ) at (0,-0.5) {$G(z)$};
\node[bdblock] (KZ) at (0.625,-2.75) {$K_0(z)$};
\node[bdblock] (FILT) at (-0.625,-2.75) {$z^{-1}I$};




\path(KZ) -- +(1,0) coordinate (y2);
\draw[bdarrow] (GZ) -- ++(1,0) coordinate (y1) -- ++(0.75,0) |- (y2) -- (KZ);



\draw[bdarrow] (KZ) -- (FILT);
\draw[bdarrow] (FILT) -- ++(-1.125,0) |- (GZ);


\begin{pgfonlayer}{background}
\draw[thick, dashed, rounded corners](-1.5,0.5) rectangle (1.5,-1.5);
\draw[thick, dashed, rounded corners](-1.5,-1.75) rectangle (1.5,-3.75);

\node[above left] at (1.5,-1.5){$\tilde{G}(z)$};
\node[below left] at (1.5,-1.75){$\tilde{K}_0(z)$};
\end{pgfonlayer}
\end{tikzpicture}
}\hfill\hfill
\caption{Techniques to ensure a strictly proper $\tilde{K}_0(z)$}
\label{fig:loopshift}
\end{figure}
However, the direct feedthrough component incorporated into the plant model uses the measured output, whilst an MPC prediction would use the observer output (the output values cannot be extrapolated over the prediction horizon without the estimates of unmeasured states).  Input constraints might therefore be violated, or, control can be overly conservative when the observer error $y(k) - C\hat{x}(k|k-1)$ is large in magnitude.

\subsubsection{Unit delay or low-pass filter}
Alternatively, adding a unit delay or a low pass filter in series with the original controller (Figure~\ref{fig:loopshift}b) prior to obtaining the observer-based realisation would have the desired effect.  By avoiding a direct feedthrough, the MPC controller directly manipulates the plant input $u(k)$ rather than an estimate of the input, avoiding uncertainty of the ``real'' value of input $u(k)$.  In this case, the conventional controller $K_0(z)$ must be sufficiently robust, or the sampling frequency must be high enough for the delay to be tolerated.

\subsection{Ensuring correct zeros in controller}
When using a filter structure, $K_0(0) = 0$ is required for correct reproduction.  If this is not initially the case, the required zeros can be artificially introduced by adding a dipole on each channel of the form
\begin{equation}
\frac{Wz}{Wz-1}
\end{equation}
where $W$ is a ``large'' number, into the open-loop controller model.  This introduces the required zeros whilst at the same time has minimal effect on open loop gains and phase shifts of the unconstrained controller (Figure~\ref{fig:dipole}).

\begin{figure}[h]
\centering
\includegraphics[width=0.7\textwidth]{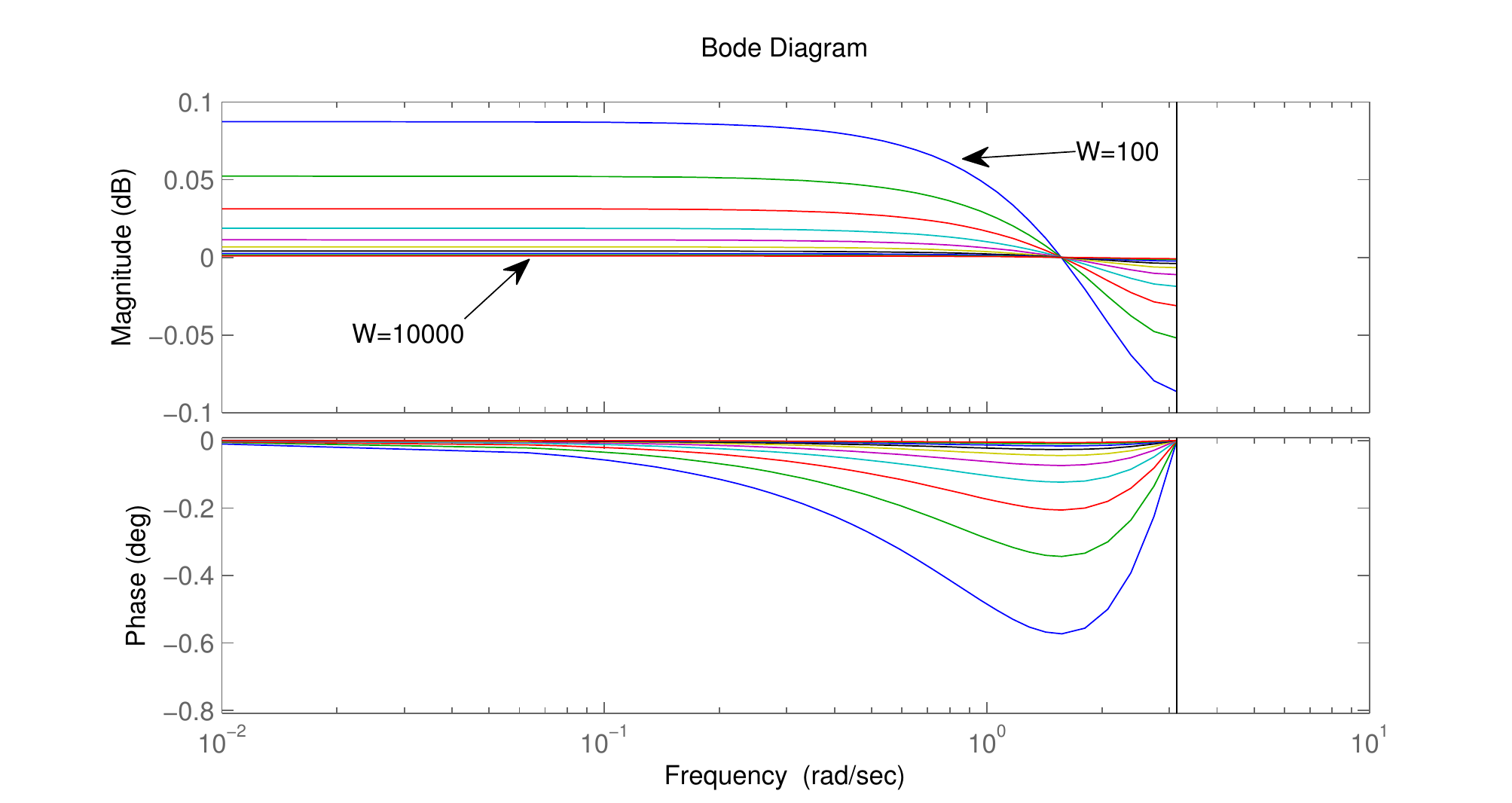}
\caption{Bode plot of dipole gain-phase properties ($T_s=\unit{1}{\second}$)}
\label{fig:dipole}
\end{figure}

\subsection{Design guidelines}

The choice as to whether the predictor and filter form is most suitable depends upon whether the original controller is strictly proper or not, and if it is not, how large the value of $D_K$ is, the length of the sampling period and the computational budget.  The type of disturbances expected should also be considered.  Loop-shifting might be considered appropriate if sensor noise dominates the model uncertainty.  On the other hand if large external disturbances act on the plant, the error between the observer estimate of the output and the actual output could make enforcement of input constraints rather difficult, making the filter structure a rather more attractive prospect.

\section{Plant model transformations}

\subsection{Integral action}
Direct reproduction of a controller including integral action for offset-free control through the reverse engineering procedure would, by construction, reproduce the input/output characteristics.  However, the presence of any unmodelled disturbance would manifest itself as a bias on each state estimate --- problematic for a constrained predictive controller, as poor predictions made from biased state estimates will lead to overly conservative control action, or in the worst cases, control action that leads to infeasibility.  The usual MPC methods of augmenting the plant model with a disturbance model \citep{MB2002,Pan2004,PB2007} can be used, subject to the new, augmented plant model being observable.
Disturbance models also provide a convenient way in which the order of the plant model can be increased when of lower order than the original controller.

Whilst, in \citep{AA1999} a method is proposed for finding an observer-compensator-Youla Parameter realisation of a controller with order higher than that of the plant model, the inclusion of disturbance models will improve the quality of the state estimates, and therefore the quality of the predictions in the optimisation.  It is noted that there has been recent interest in using a Youla Parameter as a means of improving the robustness in constrained MPC \citep{CKCR2009,TNP2010}, however the applicability of these methods to the reverse engineering procedure remains an open topic for investigation.

\section{Selection of $T$ and $T^{\dagger}$}
The allocation of the closed loop poles between the observer error and the (unconstrained) state feedback dynamics is an important design decision when implementing an MPC controller in this manner.  This turns out to be even more important when there is plant-model mismatch as a result of modelling error, or of linearisation error stemming from the common practice of using a local linearised model of a non-linear plant.

\subsection{Solution $T$}
There can be a marked difference in the observer error dynamics of different realisations, despite the complete controller being identical to the original $K_0(z)$.  As previously stated, in the presence of constraints, this error can result in a violation of constraints, or overly conservative control (depending on the sign of the error),  because $D_k y$ is directly fed back to the input of the plant, bypassing the MPC controller, whilst the MPC controller has to enforce constraints using an estimate, $C\hat{x}$ from the observer.

These types of error are particularly marked when using loop-shifting to make $\tilde{K}_0(z)$ strictly proper.
For useful predictions of the state trajectory to the obtained, the quality of the state estimation must be sufficiently high.  Whilst intuition would suggest (subject to existence of a valid solution of $T$) that keeping the fastest closed-loop poles in the observer error dynamics might be sensible, our third example (Section~\ref{sec:b747}) demonstrates that this is not necessarily the case for highly-coupled MIMO plants, particularly when using a highly-coupled MIMO plant model obtained by the widespread method of locally linearising a nonlinear plant.

For SISO systems, and small MIMO systems, it is practical to evaluate every feasible solution of $T$ and analyse the observer performance ``by eye''.  However (assuming full observability and controllability of $G_0$ and $K_0$, no repeated poles (in which case, special care must be taken as to how the subspaces are partitioned, or all repeated poles must remain ``together''), and no conjugate pole pairs), there are up to ${}^{(n+n_K)}C_{n_K}$ possible solutions for $T$.  The combinatorial growth of possibilities with the system size therefore motivates the suggestion of a fast-to-calculate quantitative metric to rapidly assess the quality of the observer.

\begin{prop}
\label{prop:Ginf}
A good observer-gain realisation to use as the basis for an MPC controller is that which minimises $\| G_{y\rightarrow \hat{e}}(z) \|_{2} \times \| G_{d \rightarrow \hat{x}}(z)\|_{2}$
where, depending on the observer realisation
\begin{equation}
G_{y\rightarrow \hat{e}}(z)
 =
\underbrace{
\left[
\begin{array}{c|c}
\tilde{A}-K_f\tilde{C} & K_f\\
\hline
C & -I
\end{array}
\right]}_{\text{Predictor form}}
\text{ or }
\underbrace{\left[
\begin{array}{c|c}
A(I-K_fC) & AK_f\\
\hline
C(I-K_fC) & CK_f-I
\end{array}
\right]}_{\text{Filter form}}
\end{equation}
and
\begin{equation}
G_{d\rightarrow\hat{x}}(z) =
\left[
\begin{array}{cc|c}
\multicolumn{2}{c|}{\multirow{2}{1.3cm}{$\tilde{A} - K_f\tilde{C}$}} & 0\\
& & I\\
\hline
I & 0 & 0\\
0 & I & -I\\
\end{array}
\right]%
\text{ or }
\left[
\begin{array}{cc|c}
\multicolumn{2}{c|}{\multirow{2}{1.75cm}{$A(I-K_fC)$}} & 0\\
& & I\\
\hline
I & 0 & 0\\
0 & I & -I\\
\end{array}
\right].
\end{equation}
\end{prop}
\begin{rmk}
The term $G_{y \rightarrow \hat{e}}$ is concerned with the effect of measurement noise on the filtered estimate of the measured outputs.  Unmeasured states are deliberately omitted in this metric, as no assumption can be made regarding the actual values of the unmeasured states, and therefore on the state error.  The term $G_{d \rightarrow \hat{x}}$ considers the effect of unmeasured (but acknowledged, if not modelled in detail) disturbances on the estimated state.  Ideally this value should be small.  The two terms are multiplied rather than added, because no assumption can be made regarding their relative magnitudes.  The $\mathcal{H}_2$ norm is chosen over the $\mathcal{H}_{\infty}$ norm because the search is being performed over a discrete set, and the latter metric represents a ``worst case'' gain, effectively ``hiding'' any other behaviour.
\end{rmk}
\begin{rmk}
The choice of $T^{\dagger}$ will affect the chosen metric.  Therefore, a simple heuristic for choosing this should be decided before performing the search over all feasible combinations $T$.  This is provided subsequently.
\end{rmk}

\subsection{Designing $T^{\dagger}$}
The extra dynamics introduced as a consequence of the observer being of higher order than the original controller affect the MPC controller performance, despite their associated modes being in the nullspace of the initial calculated $K_c$ and thus ``invisible'' to the plant (in which case ensuring $A_{E}$ is stable is sufficient \citep{DAC2006}).  The solution to a linearly constrained MPC problem with linear or quadratic cost is piecewise affine with respect to the current state \citep{BMDP2002}.  When constraints are active the gain component of this function changes, and the observer error modes which were previously invisible in the closed loop system will stop being insignificant.

The examples in Section~\ref{sec:casestudy} indicate that placing these ``free poles'' using Kalman filter methods on system (\ref{eqn:freepolesystem}) rather than attempting to place them near the origin as might be expected to give ``fastest'' convergence is worthwhile to avoid amplification of noise.  As one would reasonably expect, the choice of weightings ultimately depends on the disturbances that are expected, but a relatively high measurement noise covariance matrix (giving slower observer poles) appears to give the best results.

\section{Reference tracking}
\label{sec:reftrack}
\begin{figure*}[htbp]
\centering
\scriptsize
\hfill
\subfigure[Error observer]{
\begin{tikzpicture}
\tikzstyle{bdblock} = [draw,  thick, minimum height=0.7cm, minimum width=0.7cm];
\tikzstyle{sumblock} = [draw, circle,  thick, minimum height=0.35cm, minimum width=0.35cm];
\tikzstyle{bdarrow} = [draw, ->,  thick];
\node[bdblock] (GZ1) at (0,0) {$G(z)$};
\coordinate (u1) at (-0.625,0);
\node[bdblock] (MPC1) at (-1.25,0) {MPC};
\node[bdblock] (OBS1) at (-2.5,0) {Obs};
\node[sumblock] (S1) at (-3.5,0){};

\path(S1) -- +(-0.75,0) coordinate(r1);
\draw[bdarrow](r1) -- (S1);
\draw[bdarrow](S1) -- (OBS1);
\draw[bdarrow](OBS1) -- (MPC1);
\draw[bdarrow](MPC1) -- (GZ1);

\draw[bdarrow] (u1) -- ++(0,-0.75) -| (OBS1);
\draw[bdarrow] (GZ1) -- ++(0.75,0) -- ++(0,-1.25) -| (S1);
\end{tikzpicture}
}
\hfill
\subfigure[Option 2]{
\begin{tikzpicture}
\tikzstyle{bdblock} = [draw,  thick, minimum height=0.7cm, minimum width=0.7cm];
\tikzstyle{sumblock} = [draw, circle,  thick, minimum height=0.35cm, minimum width=0.35cm];
\tikzstyle{bdarrow} = [draw, ->,  thick];
\node[bdblock] (GZ2) at (0,0) {$G(z)$};
\node[bdblock] (OBS2) at (-0.75,-1) {Obs};
\node[bdblock] (MPC2) at (-1.5,0) {MPC};
\draw[bdarrow] (GZ2) -- ++(0.75,0) |- (OBS2);

\path (MPC2.west) -- ++(0,0.125) coordinate(pr) -- ++(-0.5,0) coordinate(pri);
\path (MPC2.west) -- ++(0,-0.125) coordinate(px) -- ++(-0.5,0) coordinate(pxi);

\draw[bdarrow] (OBS2) -| (pxi) -- (px);
\draw[bdarrow] (pri) -- (pr);
\draw[bdarrow] (MPC2) -- (GZ2);

\coordinate (u1) at ($(MPC2)!0.5!(GZ2)$);
\draw[bdarrow] (u1) -- (OBS2);

\end{tikzpicture}
}
\hfill
\subfigure[Option 3]{
\begin{tikzpicture}
\tikzstyle{bdblock} = [draw,  thick, minimum height=0.7cm, minimum width=0.7cm];
\tikzstyle{sumblock} = [draw, circle,  thick, minimum height=0.35cm, minimum width=0.35cm];
\tikzstyle{bdarrow} = [draw, ->,  thick];
\node[bdblock] (GZ2) at (0,0) {$G(z)$};
\node[bdblock] (OBS2) at (-0.75,-1) {Obs};
\node[bdblock] (MPC2) at (-1.5,0) {MPC};
\node[bdblock] (KPRE) at (-2.75,0.125) {$H_{\mathrm{pre}}(z)$};
\draw[bdarrow] (GZ2) -- ++(0.5,0) |- (OBS2);

\path (KPRE.west) -- ++(0,0) coordinate(kpr) -- ++(-0.5,0) coordinate(kpri);

\path (MPC2.west) -- ++(0,0.125) coordinate(pr) -- ++(-0.5,0) coordinate(pri);
\path (MPC2.west) -- ++(0,-0.125) coordinate(px) -- ++(-0.25,0) coordinate(pxi);

\draw[bdarrow] (OBS2) -| (pxi) -- (px);

\draw[bdarrow] (kpri) -- (kpr);

\draw[bdarrow] (KPRE) -- (pr);

\draw[bdarrow] (MPC2) -- (GZ2);

\coordinate (u1) at ($(MPC2)!0.5!(GZ2)$);
\draw[bdarrow] (u1) -- (OBS2);

\end{tikzpicture}
}
\hfill\hfill
\caption{Reference tracking}
\label{fig:reftrack}
\end{figure*}
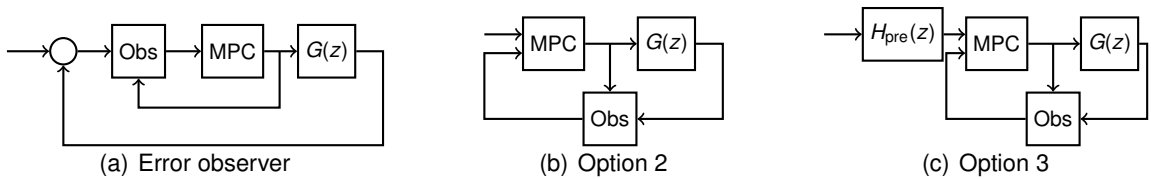
An LTI compensator is often placed in the forward path of a feedback loop rather than the return path.  The dynamics of the compensator therefore act upon the difference between the reference signal and the output signal rather than just the output signal.  However, to obtain an estimate of the plant state an observer would be placed in the return path.

Three options are suggested here.  The first is to implement the observer in the forward path, and use the observer to estimate the error between the plant state and a reference state.  Because the prediction model in the MPC controller would therefore predict the trajectory of the tracking error rather than the plant state, this is sufficient for handling input constraints, but not output or state constraints.

The second option is to simply implement the observer in the return path and accept a change in the transient response to input changes (the disturbance rejection properties will remain unchanged).

A third solution is to implement the observer in the return path, and to pre-filter the output reference set-point with a modified copy of the observer.  For the unconstrained case with unmodified controller gain $K_c$, an adequate pre-filter for the predictor form (assuming loop-shifting has not been used) is:
\begin{equation}
H_{\mathrm{pre}}(z)
=
\left[
\begin{array}{c|c}
\tilde{A}-K_f\tilde{C} & K_f\\
\hline
I & 0
\end{array}
\right].
\label{eqn:hpre}
\end{equation}
Loop-shifting must also be reflected in the reference-tracking structure (Figure~\ref{fig:lsreftrack}). The signal $D_Kr$ must be known to the MPC controller to make correct predictions.
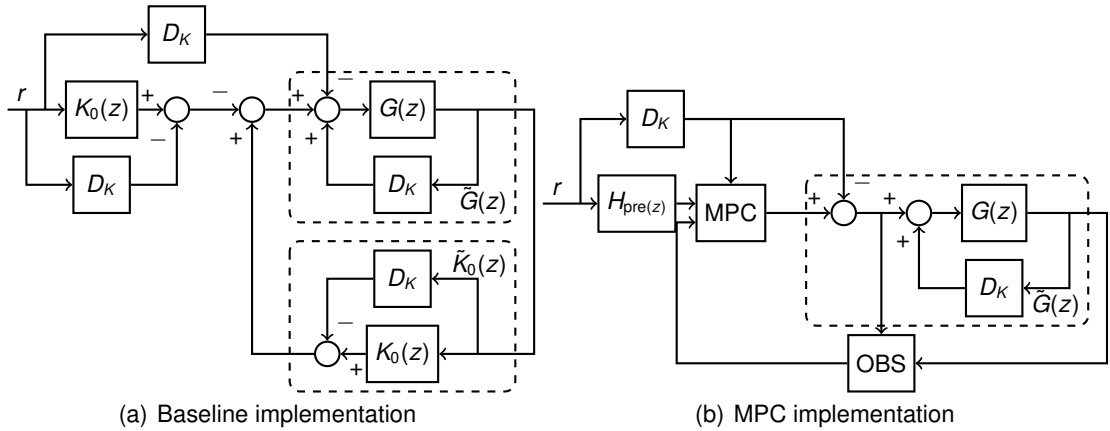
\begin{figure}[htbp]
\centering
\footnotesize
\subfigure[Baseline implementation]{
\begin{tikzpicture}
\tikzstyle{bdblock} = [draw, minimum height=0.75cm, minimum width=0.75cm,  thick];

\tikzstyle{sumblock} = [draw, minimum height=0.325cm, minimum width=0.325cm,  thick, circle];

\tikzstyle{bdarrow} = [draw,  thick, ->];

\node[bdblock] (GZ) at (0,0) {$G(z)$};
\node[bdblock] (KZ) at (0,-3.25) {$K_0(z)$};

\node[bdblock] (DK1) at (0,-1) {$D_K$};
\node[bdblock] (DK2) at (0,-2.25) {$D_K$};

\node[sumblock] (S1) at (-1,0){};
\node[sumblock] (S2) at (-1,-3.25){};

\node[sumblock] (S3) at (-2,0){};
\node[sumblock] (S4) at (-3,0){};
\node[bdblock] (KZ2) at (-4,0){$K_0(z)$};
\node[bdblock] (DK3) at (-4,-1){$D_K$};
\node[bdblock] (DK4) at (-3,1){$D_K$};

\coordinate (r1) at (-5.25,0);

\node[above right] at (r1) {$r$};
\draw[bdarrow] (r1) -- ++(0.25,0) coordinate(r2) --  ++(0.25,0) coordinate (r3) -- (KZ2);
\draw[bdarrow] (r2) |- (DK3);
\draw[bdarrow] (r3) |- (DK4);
\draw[bdarrow] (DK4) -| node[above right,pos=1]{$-$} (S1);

\draw[bdarrow] (DK1) -| node[pos=1, below left]{$+$} (S1);
\draw[bdarrow] (S1) -- (GZ);

\path(KZ) -- +(1,0) coordinate (y2);
\draw[bdarrow] (GZ) -- ++(1,0) coordinate (y1) -- ++(0.75,0) |- (y2) -- (KZ);
\draw[bdarrow] (y1) |- (DK1);

\draw[bdarrow] (DK2) -| node[pos=1, above right]{$-$} (S2);
\draw[bdarrow] (KZ) -- node[pos=1, below right]{$+$} (S2);

\draw[bdarrow] (S2) -- ++(-0.75,0) -| node[pos=1, below left]{$+$} (S3);
\draw[bdarrow] (S3) -- node[pos=1, above left]{$+$} (S1);

\draw[bdarrow] (KZ2) -- node[pos=1,above left]{$+$} (S4);
\draw[bdarrow] (DK3) -| node[pos=1,below left]{$-$} (S4);
\draw[bdarrow] (S4) -- node[pos=1, above left]{$-$} (S3);

\draw[bdarrow] (y2) |- (DK2);

\begin{pgfonlayer}{background}
\draw[thick, dashed, rounded corners](-1.5,0.5) rectangle (1.5,-1.5);
\draw[thick, dashed, rounded corners](-1.5,-1.75) rectangle (1.5,-3.75);

\node[above left] at (1.5,-1.5){$\tilde{G}(z)$};
\node[below left] at (1.5,-1.75){$\tilde{K}_0(z)$};
\end{pgfonlayer}
\end{tikzpicture}}
\subfigure[MPC implementation]{\begin{tikzpicture}

\tikzstyle{bdblock} = [draw, minimum height=0.75cm, minimum width=0.75cm,  thick];

\tikzstyle{sumblock} = [draw, minimum height=0.325cm, minimum width=0.325cm,  thick, circle];

\tikzstyle{bdarrow} = [draw,  thick, ->];

\node[bdblock] (GZ) at (0,0){$G(z)$};
\node[bdblock] (DK1) at (0,-1){$D_K$};
\node[sumblock] (S1) at (-1,0) {};
\node[sumblock] (S2) at (-2,0) {};
\node[bdblock] (DK2) at (-4.5,1.25){$D_K$};
\node[bdblock] (MPC) at (-3.5,0){MPC};
\node[bdblock] (HPRE) at (-4.75,0.125) {$H_{\mathrm{pre}(z)}$};
\node[bdblock] (OBS) at (-1.5,-2) {OBS};
\coordinate (r1) at (-6,0.125);

\draw[bdarrow] (DK1) -| node[pos=1, below left]{$+$}(S1);
\draw[bdarrow] (DK2) -| (MPC);
\draw[bdarrow] (S1) -- (GZ);
\draw[bdarrow] (S2) -- coordinate[pos=0.5](u1) node[pos=1,above left]{$+$} (S1);
\draw[bdarrow] (GZ) -- ++(1,0) coordinate (y1) -- ++(0.5,0) |- (OBS);
\draw[bdarrow] (y1) |- (DK1);
\draw[bdarrow] (DK2) -| node[pos=1, above right]{$-$} (S2);

\draw[bdarrow] (r1) -- node[pos=0,above right]{$r$} ++(0.5,0) coordinate (r2) -- (HPRE);
\draw[bdarrow] (r2) |- (DK2);
\draw[bdarrow](HPRE) -- (HPRE-|MPC.west);
\draw[bdarrow](MPC) -- node[pos=1, above left]{$+$} (S2);

\draw[bdarrow](u1) -- (OBS);

\path(MPC.west) -- ++(0,-0.125) coordinate (xhat) -- ++(-0.25,0) coordinate (xhatin);
\draw[bdarrow] (OBS) -| (xhatin) -- (xhat);

\begin{pgfonlayer}{background}
\draw[thick, dashed, rounded corners](-2.5,0.5) rectangle (1.25,-1.5);
\node[above left] at (1.25,-1.5){$\tilde{G}(z)$};

\end{pgfonlayer}

\end{tikzpicture}%
}

\caption{Loop-shifting with reference tracking}
\label{fig:lsreftrack}
\end{figure}
The prefilter with loop-shifting used should be:
\begin{equation}
H_{\mathrm{pre}}(z) =
\left[
\begin{array}{c|c}
A+(BD_K-K_f)C & K_f-BD_K\\
\hline
I & 0
\end{array}
\right].
\label{eqn:prefilternominal1}
\end{equation}

However, these pre-filters are not unique, because any signal in the null-space of $K_c$ can be added to the output of this system yet give identical (unconstrained) closed-loop results.  These degrees of freedom can therefore be used to force certain elements of the state reference $x_{r} = H_{\mathrm{pre}}(z) r$ to be equal to elements of the original reference signal, or to force certain elements of the state reference to always be zero.  The latter is useful if one desires that the reference setpoint not include any open-loop unstable directions.  Letting $r(k)$ be the original reference signal, and $x_{\mathrm{pre}}(k)$ be the state of system (\ref{eqn:hpre}) then the ``$C$'' and ``$D$'' matrices of (\ref{eqn:hpre}) can be chosen so that the state reference signal $x_r(k)$ satisfies (for some $L_1 \in \mathbb{R}^{(n-n_u)\times n}$, and some $L_2 \in \mathbb{R}^{(n-n_u)\times n_r}$):
\begin{equation}
\begin{bmatrix}
L_1\\
K_c
\end{bmatrix}
x_{\mathrm{r}}(k)
=
\begin{bmatrix}
0\\
K_c
\end{bmatrix}
x_{\mathrm{pre}}(k)
+
\begin{bmatrix}
L_2 \\
0
\end{bmatrix}
r(k).
\end{equation}
Therefore, assuming that $K_c$ is of full row rank, an equally valid choice of prefilter is:
\begin{equation}
H_{\mathrm{pre}}(z) =
\left[
\begin{array}{c|c}
\tilde{A}-K_f\tilde{C} & K_f\\
\hline
\begin{bmatrix}
L_1\\K_c
\end{bmatrix}^{-1}
\begin{bmatrix}
0\\K_c
\end{bmatrix}
&
\begin{bmatrix}
L_1\\K_c
\end{bmatrix}^{-1}
\begin{bmatrix}
L_2\\0
\end{bmatrix}
\end{array}
\right].
\end{equation}
Alternatively, if loop-shifting has been used,
\begin{equation}
H_{\mathrm{pre}}(z) =
\left[
\begin{array}{c|c}
A+BD_KC-K_fC & K_f-BD_K\\
\hline
\begin{bmatrix}
L_1\\
K_c
\end{bmatrix}^{-1}
\begin{bmatrix}
0\\
K_c
\end{bmatrix}
&
\begin{bmatrix}
L_1\\
K_c
\end{bmatrix}^{-1}
\begin{bmatrix}
L_2\\
0
\end{bmatrix}
\end{array}
\right].
\label{eqn:bdkprefilter}
\end{equation}

The same principle could also be applied to a prefilter implemented to provide $x_{r}(k+1|k)$ at time $k$.

\section{System realisation using standard tools}

\subsection{Pre-stabilisation}
\label{sec:prestab}
The stage cost function (\ref{eqn:stagecost}) is unusual in that it includes cross terms between the predicted state and the predicted input at each time step.  Whilst QP matrices for a finite-horizon control problem can easily be constructed manually, this structure is not always directly supported by standard MPC design and implementation software toolchains.  However, prestabilisation \citep{RKR1998} can be used.
Letting $u(k) = K_c\hat{x}(k) + \eta(k)$,
a change of coordinates transforms the stage cost (\ref{eqn:stagecost}) into
\begin{align}
\ell(x,\eta) & = 
\begin{bmatrix}
x^T & \eta^T
\end{bmatrix}
\begin{bmatrix}
0 & 0\\
0 & R
\end{bmatrix}
\begin{bmatrix}
x\\
\eta
\end{bmatrix}
\label{eqn:prestabcost}
\end{align}
and the prediction model
\begin{equation}
x(k+1) = (A+BK_c)x(k) + B\eta(k).
\end{equation}
Input constraints can then be imposed as cross-constraints between inputs and states --- i.e. as output constraints on an artificial plant model with a non-zero ``$D$'' matrix.

\subsection{Delay management}

Realising the discrete-time predictor structure in Simulink in a way that could be deployed is simple.  At time $k$, $x(k+1|k)$ should be used to calculate control action $u(k+1|k)$.  This will then be delayed by a period $T_s$ by a ``Rate Transition'' block configured for ``deterministic data transfer'' before being applied to a continuous-time plant.

From a practical perspective, to discrete-time realise a filter structure directly in Simulink requires that ``Rate Transitions''  between a continuous-time ``real world'' and the discrete-time controller are configured to not enforce deterministic data transfer, and that a transport delay is added in the ``cut'' shown in Figure~\ref{fig:delaycut}.  Unlike in an unconstrained ``observer-based'' controller, where $K_c$ is fixed and can be included directly in the observer dynamics, the control move $u$ from the MPC controller must be fed back to the observer \emph{after} calculation.

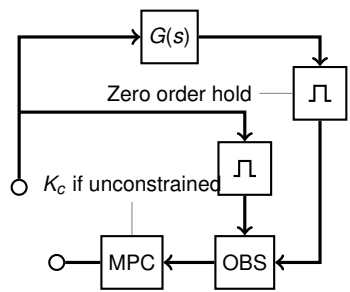
\begin{figure}[htbp]
\centering
\scriptsize
\begin{tikzpicture}
\tikzstyle{bdblock} = [  draw,  thick, minimum height=0.7cm, minimum width=0.7cm]
\tikzstyle{zohblock} = [  draw,  thick, minimum height=0.7cm, minimum width=0.7cm]
\tikzstyle{bdarrow} = [draw, very thick, ->]

\node[bdblock] at (0,0) (GS) {$G(s)$};

\node[zohblock, pin=left:{\scriptsize Zero order hold}] at (2,-0.75) (zoh) {};
\draw[thick] ($(zoh)+(-0.15,-0.1)$) -- ++(0.075,0) -- ++(0,0.2) -- ++(0.15,0) -- ++(0,-0.2) -- ++(0.075,0);

\node[bdblock] at (1,-3) (OBS) {OBS};
\node[bdblock, pin=above:{\scriptsize $K_c$ if unconstrained}] at (-0.5,-3) (MPC) {MPC};
\coordinate (est) at (-2,-1);

\node[zohblock] at (1,-1.75) (zoh2) {};
\draw[thick] ($(zoh2)+(-0.15,-0.1)$) -- ++(0.075,0) -- ++(0,0.2) -- ++(0.15,0) -- ++(0,-0.2) -- ++(0.075,0);

\draw[bdarrow] (est) |- (GS);
\draw[bdarrow] (GS) -| (zoh);
\draw[bdarrow] (zoh) |- (OBS);
\draw[bdarrow] (OBS) -- (MPC);

\draw[bdarrow] (est) -| (zoh2);
\draw[bdarrow] (zoh2) -- (OBS);

\node[draw,  thick, circle, inner sep=2pt] at ($(est)+(0,-1)$)  (cut1){};
\draw[very thick] (cut1) -- (est);

\node[draw,  thick, circle, inner sep=2pt] at ($(MPC)+(-1,0)$) (cut2){};
\draw[very thick] (MPC) -- (cut2);

\end{tikzpicture}
\caption{Delay (before observer)}
\label{fig:delaycut}
\end{figure}

Some delay to account for computation time and to avoid a computational algebraic loop is inevitable.  As well as not being a suitable configuration for controller deployment, this leads to a full unit delay on input signals fed back to the observer, because the zero-order hold sampling time will be ``missed''.  As an alternative, a delay of $T_s$ could be re-introduced everywhere, but if one has chosen to use the filter-form observer structure, it has likely already been established that this would be unacceptable.  However, to ensure deterministic data transfer whilst not requiring a full unit delay, multiple sampling rates and conditionally executed subsystems can be used (Algorithm~\ref{alg:multirate}).
\begin{algorithm}[H]
\caption{Multi-rate system for deterministic transfer}
\label{alg:multirate}
\KwData{$k$, $N_{\mathrm{div}}$}
\Begin{
\nl Let $t = kT_s$

\nl Sample $y(t)$

\nl Calculate $\hat{x}(k|k) = (I-K_fC)\hat{x}(k|k-1) + K_f y(t)$

\nl Start calculation of MPC control action

\nl $t \leftarrow kT_s + T_s/N_{\mathrm{div}}$

\nl Output MPC control action $u(t)$

\nl Use MPC control action $u(t)$ to calculate $\hat{x}(k+1|k) = A(I-K_fC)\hat{x}(k|k-1) + Bu(t) + K_fy(t-T_s/N_{\mathrm{div}})$

\nl Wait until $t = (k+1)T_s$.  Increment $k$.
}
\end{algorithm}

\FloatBarrier

\section{Cross coupling}
\label{sec:crosscoupling}
When using the reverse engineering method, under certain circumstances it is possible to inadvertently find an observer gain which introduces coupling between the state estimates of supposedly separate subsystems, and a state feedback gain that ``removes'' the cross-coupling.  Consider a closed loop system comprised of $m$ identical, parallel, decoupled loops with the plant and controller
\begin{equation}
G(z) =
\left[
\begin{array}{ccc|ccc}
A_{1} & & & B_{1} & & \\
& \ddots & & & \ddots & \\
& & A_{m} & & & B_{m}\\
\hline
C_{1} & & & 0 & & \\
& \ddots & & & \ddots & \\
& & C_{m} & & & 0\\
\end{array}
\right]
\end{equation}
\begin{equation}
K(z) =
\left[
\begin{array}{ccc|ccc}
A_{K1} & & & B_{K1} & & \\
& \ddots & & & \ddots & \\
& & A_{Km} & & & B_{Km}\\
\hline
C_{K1} & & & 0 & & \\
& \ddots & & & \ddots & \\
& & C_{Km} & & & 0\\
\end{array}
\right]
\end{equation}
respectively, where $A_1 = A_2 = \ldots = A_m$, $B_1 = B_2 = \ldots = B_m$ etc.
This is equivalent to $m$ identical, independent, closed loop systems.  The reverse engineering process aims to cast the controller into an observer form, with a state feedback matrix, $K_c$ and an observer gain matrix, $K_f$.  It would therefore not be unreasonable to expected that $K_c$ and $K_f$  to also be block diagonal:
%
\begin{align*}
K_c & = \begin{bmatrix}
K_{c1}\\
& \ddots\\
& & K_{cm}
\end{bmatrix}
&
K_f & = \begin{bmatrix}
K_{f1}\\
& \ddots\\
& & K_{fm}
\end{bmatrix}.
\end{align*}
Reverse engineering such a structure would, of course, be equivalent to reverse engineering each of the identical subsystems individually.  Unfortunately, it is not always the case that this structure is obtained, and whilst the nominal input-output characteristics remain decoupled, the reverse engineering procedure can introduce internal cross-coupling between the loops through a poor choice of $U = \begin{bmatrix}U_1^T & U_2^T\end{bmatrix}^T$ in (\ref{eqn:U1U2}).

\begin{thm}
The eigenvalues of matrix $A+BK_c$ correspond to eigenvectors defined by the columns of $U_1$, when the basis for the invariant subspace comprises a selection of eigenvectors of $A_{cl}$.
\end{thm}

\begin{proof}{}
From Theorem~\ref{thm:clpoles}, consider
\begin{equation}
A + BC_KT = A+BK_c = U_{1}\Lambda U_{1}^{-1}.
\end{equation}
If the diagonal elements of $\Lambda$ are the poles of $(A+BK_c)$, the columns of $U_1$ are the corresponding eigenvectors.
\end{proof}

\begin{lem}
\label{thm:subspace}
The solution $T$ to (\ref{eqn:filtermatch_ric}) or (\ref{eqn:predictorform_ric}), obtained using the method described in Section~\ref{sec:obsbasedreworked} is unique for a given basis $\mathrm{Im} \begin{bmatrix}U_1^T & U_2^T\end{bmatrix}^T$, and is not dependent on the scaling, nor the ordering of the columns.
\end{lem}
\begin{proof}
Let the columns be transformed by a full rank $n \times n$ matrix, $X$.
\begin{equation}
\begin{bmatrix}
\tilde{U}_{1}\\\tilde{U}_{2}
\end{bmatrix}
=
\begin{bmatrix}U_1\\U_2\end{bmatrix}
X
=
\begin{bmatrix}
U_1X\\U_2X
\end{bmatrix}
\end{equation}
Then, 
\begin{align}
\tilde{T} & = \tilde{U}_{2}\tilde{U}_{1}^{-1}
 = U_{2}XX^{-1}U_1^{-1} = T.
\end{align}
Therefore, $T$ depends on the span of the invariant subspace $\mathcal{S}$, not on its specific representation.
\end{proof}
\begin{dfn}[Eigenspace]
An eigenspace is the maximal invariant subspace corresponding to a particular eigenvalue of a matrix.
\end{dfn}
\begin{lem}
Assuming no repeated poles within each of the independent loops, there are $n= \mathrm{dim}(A_i)$ distinct eigenvalues, each of which corresponds to an $m$-dimensional eigenspace of $A_{\mathrm{cl}}$ --- i.e. any linear combination of the basis vectors that define the eigenspace is a valid eigenvector.
\end{lem}
\begin{proof}
By construction.  Consider each subsystem separately.
\end{proof}

\begin{lem}
A sufficient condition for the reverse engineered system to be decoupled is for $T$ to be block diagonal, with each block corresponding to an individual subsystem.
\end{lem}
\begin{proof}
By construction, $C_K$ is of the form
\begin{equation}
C_K =
\begin{bmatrix}
C_{K1}\\
& \ddots\\
& & C_{Km}
\end{bmatrix}.
\end{equation}
Therefore $K_c = C_KT$ is of the form
\begin{align}
K_c 
& 
= 
\begin{bmatrix}
K_{c1}\\
& \ddots \\
& & K_{cm}
\end{bmatrix}\\
& = 
\begin{bmatrix}
C_{K1}\\
& \ddots\\
& & C_{Km}
\end{bmatrix}
\begin{bmatrix}
T_{11} & \cdots & T_{1m}\\
\vdots & \ddots & \vdots\\
T_{m1} & \cdots & T_{mm}
\end{bmatrix}\\
& =
\begin{bmatrix}
C_{K1}T_{11} & C_{K1}T_{12} & \cdots & C_{K1}T_{1m}\\
C_{K2}T_{21} & C_{K2}T_{22} & \cdots & C_{K2}T_{2m}\\
\vdots & \vdots & \ddots & \vdots\\
C_{Km}T_{m1} & C_{Km}T_{m2} & \cdots & C_{Km}T_{mm}
\end{bmatrix}.
\end{align}
Therefore, $K_c$ will be decoupled when $T$ is of the form
\begin{equation}
T =
\begin{bmatrix}
T_{11}\\
& \ddots\\
& & T_{mm}
\end{bmatrix}.
\end{equation}
~
\end{proof}

\begin{rmk}
This is analogous to reverse engineering each of the loops individually.  In this case, $U_1$ and $U_2$ would also be block diagonal.  Intuitively, given that the original controller is being viewed as an observer on $Tx$, it makes sense that the loops should not affect each other.
\end{rmk}

\begin{proof}
Remembering that $T = U_{2}U_{1}^{-1}$, if $U_1$ and $U_2$ are block diagonal in a compatible fashion, then $T$ will be block diagonal:
\begin{equation}
\begin{bmatrix}
T_{11}\\
& \ddots\\
& & T_{mm}
\end{bmatrix}
=
\begin{bmatrix}
U_{21}\\
& \ddots\\
& & U_{2m}
\end{bmatrix}
\begin{bmatrix}
U_{11}^{-1}\\
& \ddots \\
& & U_{1m}^{-1}
\end{bmatrix}.
\end{equation}
This is effectively reverse engineering the decoupled systems individually.
\end{proof}
\begin{lem}
It is not necessary that $U_1$ and $U_2$ are block diagonal for $T$ to be block diagonal.
\label{thm:blockdiagnotnecessary}
\end{lem}

\begin{proof}
As proved in Lemma~\ref{thm:subspace}, the solution $T$ only depends upon the choice of invariant subspace, not its representation.  Therefore, to obtain a decoupled solution, it will suffice that there exists a matrix $X$ such that
\begin{equation}
\begin{bmatrix}
U_{1}\\
U_{2}
\end{bmatrix}
X
=
\begin{bmatrix}
\bar{U}_{11}\\
& \ddots \\
& & \bar{U}_{1m}\\
\bar{U}_{21}\\
& \ddots \\
& & \bar{U}_{2m}
\end{bmatrix}.
\end{equation}
Therefore, it is not necessary for $U_1$ and $U_2$ to have any particular structure, merely, for it to be possible to construct the desired structure through linear combinations of the columns.
\end{proof}

\begin{prop}
\label{prop:completespace}
If the invariant subspace $U$ used to solve the non-symmetric Riccati equation is constructed from a set of complete eigenspaces, inappropriate cross coupling will not be introduced.
\end{prop}

\begin{rmk}
In a decoupled system, with $m$ identical subsystems, each distinct eigenvalue will have associated with it an $m$-dimensional eigenspace.  Due to the construction of the original system, it is possible to interpret each of the $m$ dimensions as corresponding to each of the original decoupled subsystems.
Therefore, the conditions on $U_1$ and $U_2$ required in Theorem~\ref{thm:blockdiagnotnecessary} will be fulfilled automatically if the eigenspaces are not split.
\end{rmk}

\begin{prop}
If the $m$-dimensional eigenspaces are only partially used when choosing $\begin{bmatrix}U_1^T&U_2^T\end{bmatrix}^T$, cross-coupling in $K_c$ is not inevitable under some circumstances.  Whether cross-coupling is introduced depends on the manner in which the $m$-dimensional invariant subspaces are split.  Consider one of the $m$-dimensional eigenspaces, with basis
\begin{equation}
\begin{bmatrix}
| & & |\\
v_{1,1} & \cdots & v_{1,m}\\
| & & |\\
| & & |\\
v_{2,1} & \cdots & v_{2,m}\\
| & & |
\end{bmatrix} = 
\begin{bmatrix}
V_{1}\\V_{2}
\end{bmatrix}.
\end{equation}
Because $V_{1}$ is of rank $m$, there exists a $m \times m$ transformation matrix $Y$ such that
\begin{equation}
\begin{bmatrix}
| & | & & |\\
v_{1,1} & v_{1,2} & \cdots & v_{1,m}\\
| & | & & |
\end{bmatrix} Y
=
\begin{bmatrix}
\tilde{v}_{1} &  0 & \cdots & 0\\
0 & \tilde{v}_{2} &  \cdots & 0\\
0 & 0 & \cdots & 0\\
0 & 0 & \cdots & \tilde{v}_{m}
\end{bmatrix}.
\end{equation}
The columns can then be freely exchanged by postmultiplying by a permutation matrix.  A subset of the columns of $\begin{bmatrix}V_1^T & V_2^T\end{bmatrix}^TY$ can then be used to form part of the invariant subspace of $A_{cl}$, $\begin{bmatrix}U_1^T & U_2^T\end{bmatrix}^T$, without causing coupling between the nominally independent loops.
\end{prop}

\begin{thm}
If the $m$-dimensional eigenspaces are split when choosing $\begin{bmatrix}U_1^T & U_2^T\end{bmatrix}^T$, such that in neither of the resulting subspaces is it possible to obtain a decoupled structure through linear combinations of their respective bases, cross-coupling will introduced.
\end{thm}

\begin{proof}
There exists a transformation matrix $Y$ such that
\begin{equation}
\begin{bmatrix}
| & | & & |\\
v_{1,1} & v_{1,2} & \cdots & v_{1,m}\\
| & | & & |
\end{bmatrix} Y
=
\begin{bmatrix}
\tilde{v}_{1} &  0 & \cdots & 0\\
0 & \tilde{v}_{2} &  \cdots & 0\\
0 & 0 & \cdots & 0\\
0 & 0 & \cdots & \tilde{v}_{m}
\end{bmatrix}
\Gamma
\end{equation}
where $\Gamma$ is dense.
In this scenario, if fewer than $m$ columns of $\begin{bmatrix}V_1^T & V_2^T\end{bmatrix}^TY$ are used to construct $\begin{bmatrix}U_1^T & U_2^T\end{bmatrix}^T$, cross-coupling between the independent loops will be introduced if the number of columns selected from $\begin{bmatrix}V_1^T & V_2^T\end{bmatrix}^TY$ is fewer then the number of loops for which a non-zero element exists in any of the columns.
\end{proof}
\begin{rmk}
In other words, cross coupling will be inevitably introduced if any of the eigenspaces is split in such a way that neither of the resulting subspaces can be represented in a way that separates the loops.
\end{rmk}

\begin{cor}{}
\label{thm:u1u2completeig}
If $\mathrm{span}\left(\begin{bmatrix}U_1^T & U_2^T\end{bmatrix}^T\right)$ only contains complete eigenspaces of $A_{cl}$, cross-coupling will not be introduced in $K_c$.  \end{cor}
\begin{proof}
If none of the $n$, $m$-dimensional eigenspaces are split (i.e. they are used to construct $U$ in their entirety or not at all), the pole allocation between observer and feedback will be equivalent to the reverse engineering, then recombination of each of the loops individually.
Suppose that cross-coupling between the loops is introduced.  Then, there needs to be a representation of the invariant subspace of $A_{\mathrm{cl}}$,  $\begin{bmatrix}U_1^T & U_2^T\end{bmatrix}$ such that cross coupling is introduced.  However, for any given invariant subspace for which a solution $T$ exists, $T$ is unique.  A decoupled solution is known to exist, and this must, therefore, be the only solution.
\end{proof}

\section{Case Studies}
\label{sec:casestudy}
\subsection{Spacecraft Attitude Control --- Fault robustness using redundant actuators}
\label{sec:scattitude}
This example demonstrates the procedure on the sampled-data output-feedback controller from \cite[ex 9.4.1]{Sidi1997} for a single axis attitude control system with angle measurement only, and shows how the implicit daisy-chaining of constrained MPC \citep{Mac1998} can be exhibited by a reverse-engineered controller.  A continuous-time linear model is used for simulation, and its zero-order-hold discretisation used for prediction.

\subsubsection{Model}
The plant is modelled as an ideal inertial load with a moment of inertia $J = \unit{500}{\kilo\gram~\meter^{-2}}$.  To demonstrate how reverse-engineered MPC can be used to add extra functionality to the original controller, a redundant torque pair input is added, and the model is also augmented with a constant torque disturbance state.  The sampling period $T_s = \unit{0.25}{\second}$.
\begin{equation}
G(z) =
\left[
\begin{array}{ccc|cc}
1 & 0.25 & 0.00358 & 0.00358 & 0.00358\\
0 & 1 & 0.02865 & 0.02865 & 0.02865\\
0 & 0 & 1 & 0 & 0\\
\hline
0.01745 & 0 & 0 & 0 & 0
\end{array}
\right]
\end{equation}

\subsubsection{Baseline controller}
The baseline controller does not use the second torque pair,
\begin{equation}
K_0(z) =
\left[
\begin{array}{cc|c}
1.412 & -0.8235 & 32\\
0.5 & 0 & 0\\
\hline
13.01 & -26.14 & -871\\
0 & 0 & 0
\end{array}
\right]
\end{equation}
and has poles at $1$ (integral action), $0.41$ and zeros at $0.98$ and $0.91$.  

\subsubsection{Observer-compensator realisation}
The magnitude of the $D_K$ term is large in comparison to other terms in the controller matrices, and the controller is being designed to counteract external disturbances, so the estimate of $y$ may have significant error during transients caused by these disturbances.  Figure~\ref{fig:attitudeloopshiftobserver} shows the observer error dynamic responses for each of the possible realisations when loop-shifting is used --- it is clear that the error is substantial.  Loop shifting is therefore not considered an option for this system if input constraints are to be enforced.  Adding a unit delay is also not an option because it noticeably changes the system response step response to the disturbance (Figure~\ref{fig:attitudecompareperformance1}).

\begin{figure*}[htbp]
\includegraphics[width=\textwidth]{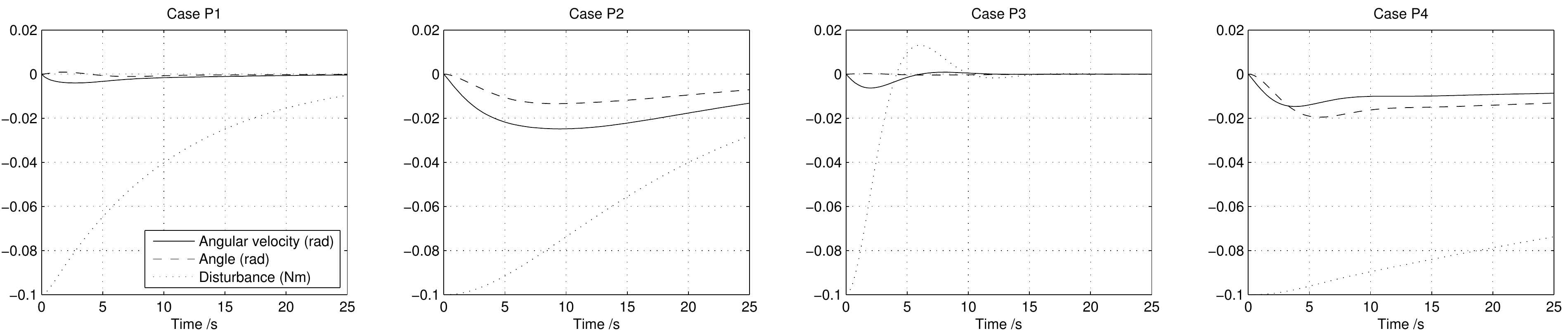}
\caption{Observer error dynamics using loop-shifting}
\label{fig:attitudeloopshiftobserver}
\end{figure*}

\begin{figure}[htbp]
\centering
\includegraphics[width=0.475\textwidth]{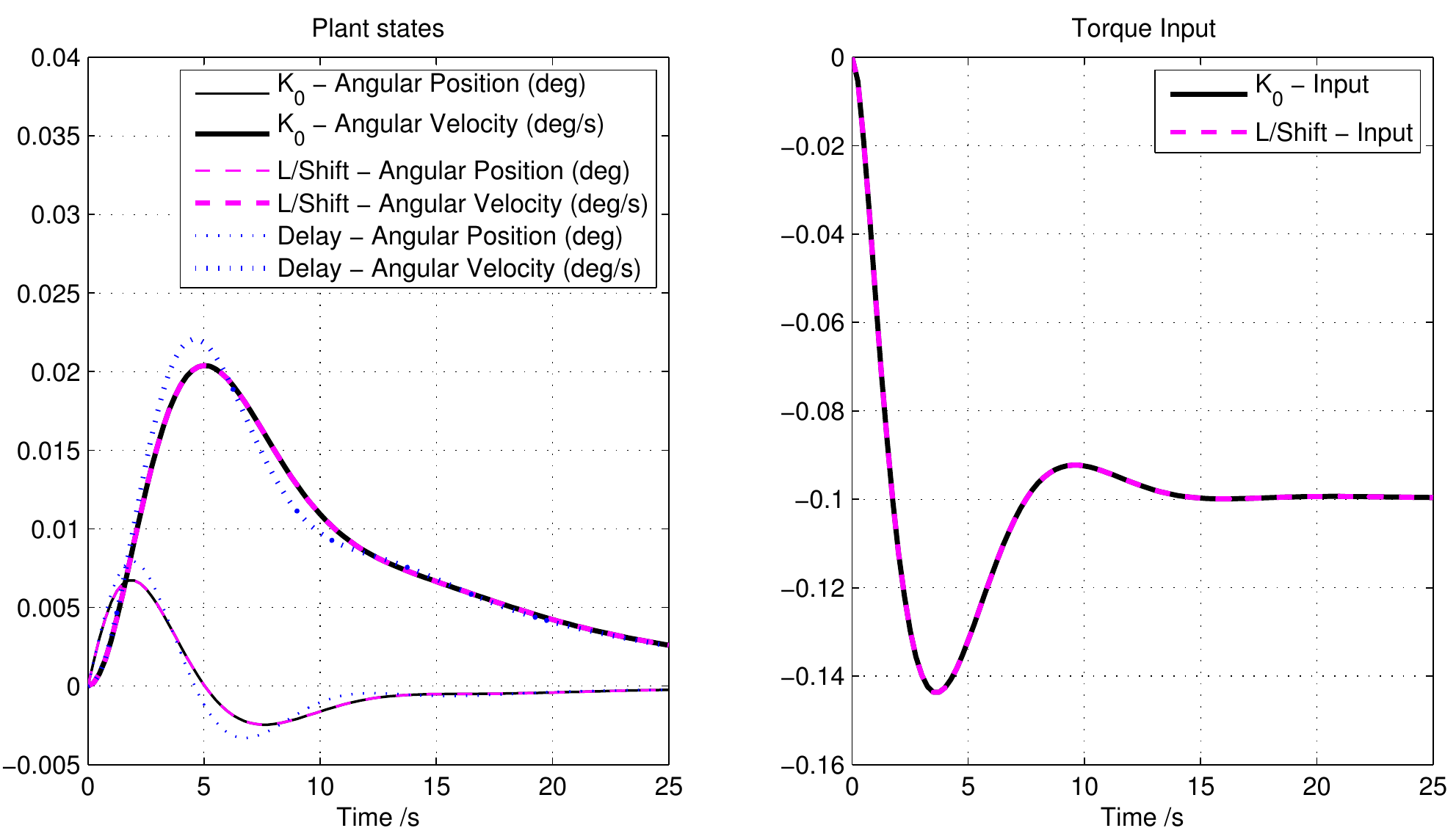}
\caption{Comparison of loop shifting, unit delay and $K_0$ closed loop performance}
\label{fig:attitudecompareperformance1}
\end{figure}
The remaining option is to obtain a filter-form observer-compensator realisation.  Whilst $K_0(0) \ne 0$, so a dipole can be added to make this condition hold:
\begin{equation}
K_1(z) = K_0(z) \frac{50z}{50z-1}.
\end{equation}
Reassuringly, the closed loop poles do not move significantly (Table~\ref{tbl:attclpoles}).%
\begin{table}[h]
\caption{Closed-loop poles for unconstrained attitude control system}
\label{tbl:attclpoles}
\footnotesize
\centering
\renewcommand{\arraystretch}{1.25}
\begin{tabular}{cc}
\hline
\hline
Original & With Dipole\\
\hline
$1$ & $1$\\
$0.9764$ & $0.9764$\\
$0.9115 \pm \mathrm{j}0.1192$ & $0.9086 \pm \mathrm{j}0.1204$\\
0.5579 & $0.5660$\\
& $0.0177$\\
\hline
\hline
\end{tabular}
\end{table}
There are three plant states (including the disturbance), and, now, three controller states. The disturbance state is uncontrollable, and the complex pole pair cannot be split, so there are four possible choices of $T$ giving the realisations in Table~\ref{tbl:attobscombos}---``S'' is used to mean a pole placed in the state feedback, and ``O'' the observer error dynamics.  In reality there will be a small transport delay between measurement and application of the control action due to computation not being instantaneous.  For this reason, the phase margin and the delay margin, having broken the loop just after the MPC calculation (but before feeding back to the observer and plant) are considered.
\begin{table}[h]
\caption{Observer-based realisations for attitude system}
\footnotesize
\label{tbl:attobscombos}
\centering
\renewcommand{\arraystretch}{1.25}
\begin{tabular}{ccccc}
\hline
\hline
 & R1 & R2 & R3 & R4\\
\hline
$0.0177$ & O & S & S & O\\
$0.5660$ & O & S & O & S\\
$0.9086 + \mathrm{j}0.1204$ & S & O &O &O\\
$0.9086 - \mathrm{j}0.1204$ & S & O &O &O\\
$0.9764$ & O & O & S & S\\
$1$ & S & S & S & S\\
\hline\hline
Delay margin & $4.36T_s$ & $0.51T_s$ & $0.98T_s$ & $2.83T_s$\\
Gain margin & $11.67$ & $1.66$ & $2.01$ & $4.42$\\
\hline\hline
$\| G_{y\rightarrow \hat{e}}(z) \|_{2}$ & $29.95$ & $12.31$ & $18.89$ & $89.05$\\
$\| G_{d \rightarrow \hat{x}}(z)\|_{2}$ & $5.03$ & $5.59$ & $3.17$ & $2.99$\\
$\|G_{y\rightarrow \hat{e}}\|_2\times \|G_{d\rightarrow\hat{x}}\|_2$ & $150.5319$ & $68.7844$ & $59.8672$ & $89.0512$\\
\hline\hline
\end{tabular}
\end{table}

Superficially, the data from Table~\ref{tbl:attobscombos} suggests that realisation R1 is a sensible starting point for development of an MPC controller.  The fastest poles are in the observer error dynamics, and the delay margin and the gain margin at the artificial ``cut'' are excellent.  Furthermore, if implemented using a single discrete sampling rate (e.g. a basic implementation in Simulink), realisations R2 and R3 become unstable because the inevitable delay causes a zero-order hold sampling time to be ``missed'' introducing a whole unit delay.

However, Proposition~\ref{prop:Ginf} suggests something rather different.  In fact the observer in R1 is very slow to estimate unmeasured disturbances (resulting in large errors in the estimates of other states).  R4 does not have this issue, however, it amplifies measurement noise (albeit absent in this example) in an unpleasant manner.  Therefore, the best overall realisation is R3.  If the information flow is managed carefully (e.g. using multiple sampling rates in Simulink, as proposed in Algorithm~\ref{alg:multirate}), the computation delay need only be a small fraction of the sampling period, and can be moved to ``after'' the point at which the calculated control action has been fed back to the observer, but before the plant---thus recovering the behaviour of the original controller even for R2 and R3, and rendering the calculated ``delay margin'' at the original ``cut'' (Figure~\ref{fig:delaycut}) irrelevant.

\subsubsection{MPC implementation -- Cost function and constraints}
Table~\ref{tbl:attcases} enumerates the controller configurations used to produce the closed loop responses to the step disturbance depicted in Figure~\ref{fig:attcases}.  When cost function (\ref{eqn:stagecost}) is used, $R=I$.  When cost function (\ref{eqn:stattrack}) is used, $R_1=10^{-3}I$ and $Q_1=10^3I$ to indicate a strong weighting for matching the ``effect'' of the control action on the state and a weak weighting on the exact original control configuration.

\begin{table}[htbp]
\caption{Reverse engineered MPC with input constraints}
\label{tbl:attcases}
\footnotesize
\renewcommand{\arraystretch}{1.25}
\centering
\begin{tabular}{ccccc}
\hline\hline
&  $K_0$ & Case 1 & Case 2 & Case 3\\
\hline
Realisation & $K_0$ & R3 & R3 & R3\\
Cost function & -- & (\ref{eqn:stagecost}) & (\ref{eqn:stagecost}) & (\ref{eqn:stattrack})\\
Prediction/Control Horizon & -- & 15 & 15 & 15\\
Constraints & -- & -- & $|u_i| \le 0.11$ & $|u_i| \le 0.11$\\
\hline
\hline
\end{tabular}
\end{table}

\begin{figure}[htbp]
\centering
\includegraphics[width=0.475\textwidth]{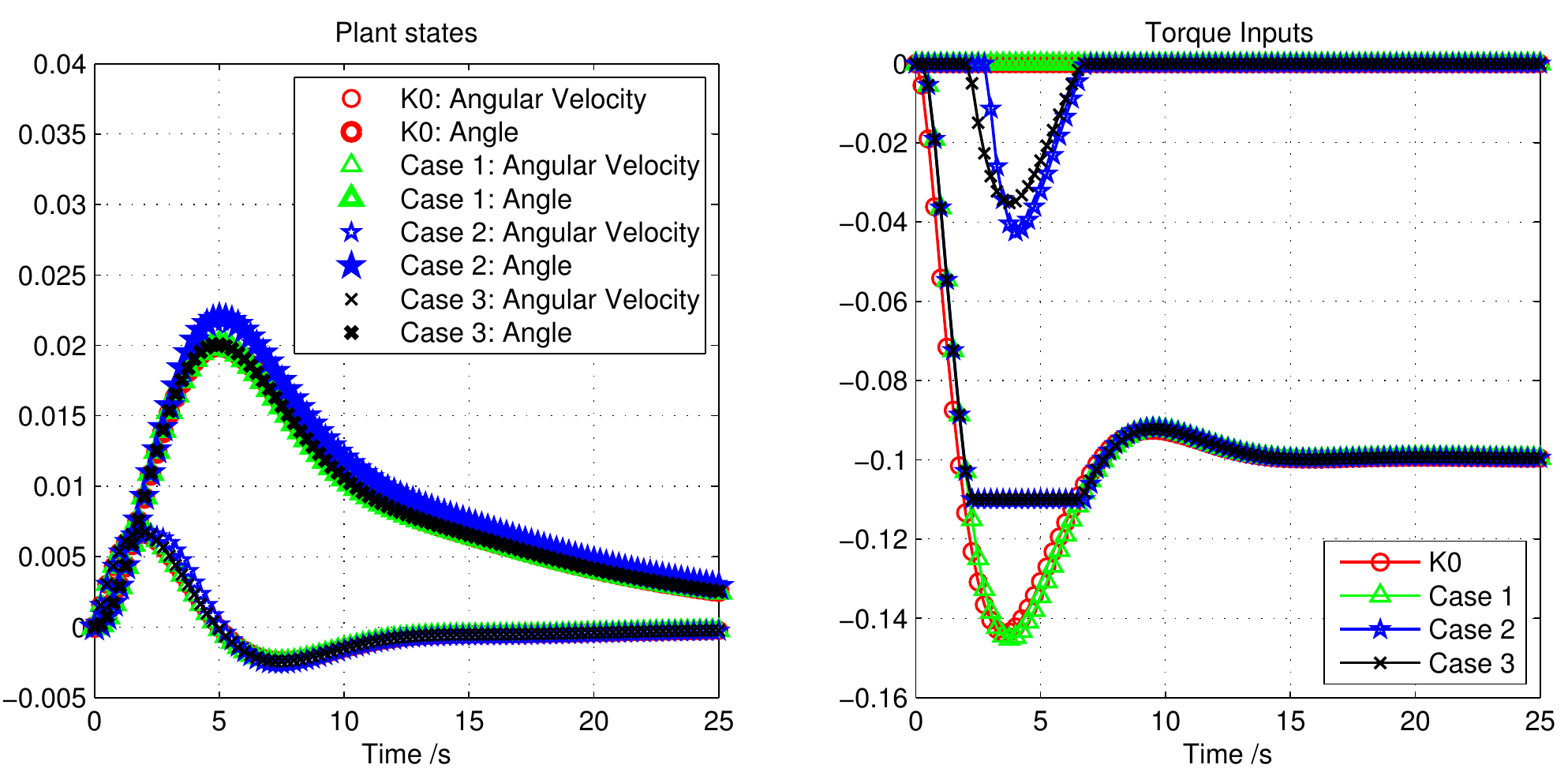}
\caption{Reverse engineered MPC with input constraints}
\label{fig:attcases}
\end{figure}

Figure~\ref{fig:attcases} shows that in unconstrained MPC Case 1, the matching with the behaviour of the original output feedback controller is close although there is a slight difference, attributed primarily to a $T_s/10$ delay introduced to ensure deterministic data transfer ($K_0$ is assumed to be implemented instantaneously).  In MPC Case 2, input constraints are deliberately imposed at a level lower than the peak of the unconstrained input trajectory.  Whilst the second torque pair can be seen to boost the correction of the disturbance, this is slightly slower than the unconstrained case.  Nevertheless, this demonstrates that the objective embedded in the cost function (\ref{eqn:stagecost}) could be adequate in this scenario.  In MPC Case 3, where the cost function better encodes the consequences of the control action rather than its specific realisation, the output trajectory is identical to the unconstrained case, with the second torque pair being used to exactly match the net torque trajectory of the unconstrained case.

\subsubsection{MPC implementation -- Output constraints and fault recovery}
In MPC Cases 4 and 5, an output constraint is also added to constrain the angle $y \le 0.01$, representing an attitude pointing requirement.  To ensure feasibility of the optimisation problem despite observer estimation error, these are softened with a quadratic weighting of $10^5$.  In Case 4, the input constraints are relaxed, and in Case 5, an unmodelled plant failure is introduced at $t=\unit{3}{\second}$.  

\begin{table}[htbp]
\caption{Reverse engineered MPC with state constraints}
\label{tbl:attcasesoutput}
\footnotesize
\renewcommand{\arraystretch}{1.25}
\centering
\begin{tabular}{ccccc}
\hline\hline
&  $K_0$ & Case 4 & Case 5 \\
\hline
Realisation & $K_0$ & R4 & R4\\
Cost function & -- & (\ref{eqn:stagecost}) & (\ref{eqn:stattrack})\\
Input Constraints & -- & $|u_i| \le 1$ & $|u_i| \le 0.15$\\
State Constraints & -- & $|y| \le 0.01$ & $|y| \le 0.01$\\
Torque 1 failure & -- & -- & \unit{3}{\second}\\ 
\hline
\hline
\end{tabular}
\end{table}

\begin{figure}[htbp]
\centering
\includegraphics[width=0.475\textwidth]{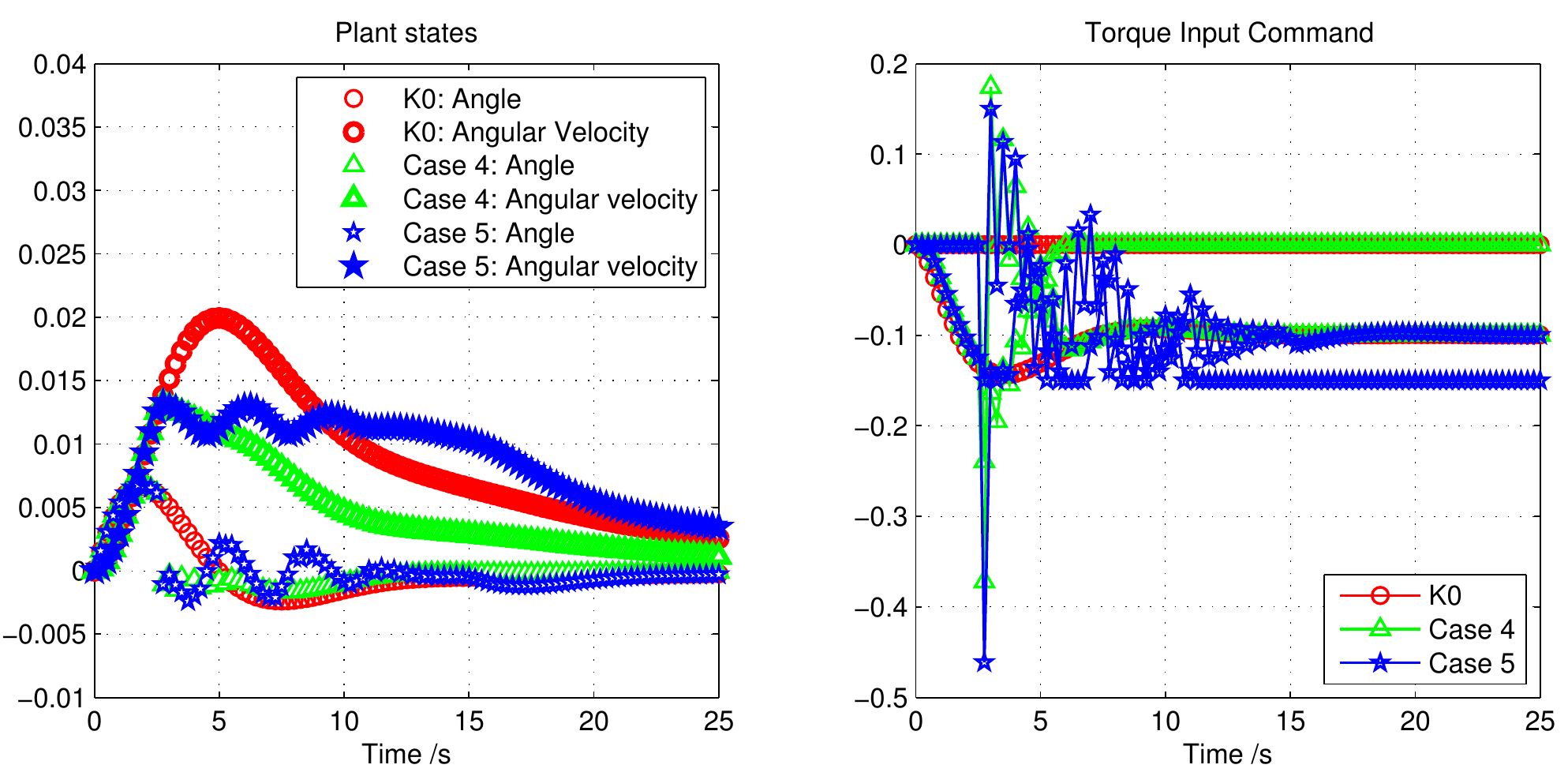}
\caption{Reverse engineered MPC with state constraints and unmodelled plant failures}
\label{fig:attcasesoutput}
\end{figure}
Because of the constraint softening and the inevitable state estimation error, constraints are not enforced exactly, however, Figure~\ref{fig:attcasesoutput} shows that the violations are not large.

\subsubsection{MPC implementation -- Tools}
The constrained reverse-engineered MPC controller is implemented in Simulink using a custom condensed QP builder that accommodates cross-terms between input and states (i.e. equality constraints for state dynamics are eliminated to form a dense QP).  The QP is solved using an Embedded MATLAB implementation of the dual active-set algorithm of \citep{GI1983}.  The prediction horizon is 15, with 4 input constraints per time step, and 2 (softened) output constraints per time step, leading to a QP with 45 decision variables (the slack variable is shared between the upper and lower bound constraints on outputs) and 90 inequality constraints.  Measured using the Simulink Profiler tool, the QP solver takes on average \unit{0.3}{\milli\second} to solve on a \unit{2.8}{\giga\hertz} Mac Pro running MATLAB R2010b on Scientific Linux 6 in a virtual machine using a single core.  This is small in comparison to the $T_s/10 = \unit{0.025}{\second}$ that has been allowed for computation in these simulations.  Assuming an approximately linear scaling with clock speed, this means that a \unit{40}{\mega\hertz} processor could be sufficient to implement the reverse-engineered constrained controller in real-time.  The subdivision of the sampling period to approximate the direct feed-through is therefore a demonstrably practical option.

%
%

\subsection{Inverted Pendulum on a Cart --- Output constraints with an unstable plant}
\label{sec:invpend}

This example demonstrates the procedure when a local linearisation of a nonlinear model is used for prediction, and shows that output constraints can be enforced as with conventional MPC design.  Unit (\unit{1}{\meter}) step changes in the cart position reference are tracked using the method presented in Section~\ref{sec:reftrack}, whilst maintaining pendulum stability.

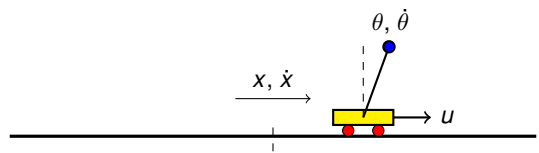
\begin{figure}[htbp]
\centering
\footnotesize
\begin{tikzpicture}
\draw[very thick](-3.5,0) -- (3.5,0);

\draw[dashed](0,-0.2) -- (0,0.2);

\begin{scope}[xshift=1cm,yshift=0.075cm]
\draw[fill=red](0,0) circle (0.075);
\draw[fill=red](0.4,0) circle (0.075);
\draw[thick,fill=yellow] (-0.2,0.075) rectangle ++(0.8,0.2);

\draw[thick] (0.2, 0.175) -- ++(70:1) node[circle,draw, inner sep=1.5pt, fill=blue]{};
\draw[dashed] (0.2, 0.175) -- node[above right,pos=1]{$\theta$, $\dot{\theta}$} ++(90:1);

\draw[thick,->](0.6,0.175) -- ++(0.5,0) node[right]{$u$};

\end{scope}

\draw[->](-0.5,0.5) -- node[above]{$x$, $\dot{x}$}++(1,0);
\end{tikzpicture}
\caption{Cart-pendulum model}
\label{fig:cpmodel}
\end{figure}

\subsubsection{Model}

A non-linear continuous-time plant model is used for simulation, whilst a linearisation about the unstable equilibrium point is used for prediction.  The model data and the baseline controller are taken from \citep{GGS2001}.  The pendulum mass $m=\unit{0.5}{\kilogram}$, cart mass $M=\unit{0.5}{\kilo\gram}$ and pendulum length $l = \unit{1}{\meter}$.  The system input is a force applied to the cart, whilst the system measured outputs are (in order) the position of the cart and the angular deviation from vertical of the pendulum.  The plant states are cart position $x$, cart velocity $\dot{x}$, pendulum deflection $\theta$ and pendulum angular velocity $\dot{\theta}$ (Figure~\ref{fig:cpmodel}).
The nonlinear model equations in state space form are:
\begin{subequations}
\begin{align}
(M+m)\ddot{x} + ml\ddot{\theta}\cos\theta - ml\dot{\theta}^2\sin\theta & = u\\
-l\ddot{\theta} + g\sin\theta & = \ddot{x}\cos\theta.
\end{align}
\end{subequations}
Rearranged into a nonlinear state-space form:
\begin{subequations}
\begin{align}
\ddot{x} & = \frac{ml \dot{\theta}^2 \sin \theta -
mg \sin \theta \cos \theta + u}{(M+m \sin^2 \theta)}\\
\ddot{\theta} & = \frac{g \sin \theta -\ddot{x}\cos \theta}{l}.
\end{align}
\end{subequations}
The continuous-time linear model is obtained by linearised about the upwards-facing equilibrium point. is:
\begin{equation}
G(s) =
\left[
\begin{array}{cccc|c}
0 & 1 & 0 & 0 & 0\\
0 & 0 & -\frac{mg}{M} & 0 & \frac{1}{M}\\
0 & 0 & 0 & 1 & 0\\
0 & 0 & \frac{(M+m)g}{Ml} & 0 & -\frac{1}{Ml}\\
\hline
1 & 0 & 0 & 0 & 0\\
0 & 0 & 1 & 0 & 0\\
\end{array}
\right].
\end{equation}

\subsubsection{Baseline linear controller}
The baseline controller is specified in continuous-time as a MISO transfer function.  Assuming positive feedback,
\begin{equation}
K_0(s) =
\begin{bmatrix}
\frac{4(s+0.2)}{(s+5)} & \frac{150(s+4)}{(s+30)}
\end{bmatrix}.
\end{equation}

\subsubsection{Reverse engineered controller}
The baseline controller $K_0(s)$ must be discretised before proceeding.  This leaves a choice for the value of $T_s$.  $K_0$ has two stable, real poles of frequency \unit{5}{\rad~\second^{-1}} and \unit{30}{\rad~\second^{-1}}.  The linearised plant model has two poles at the origin of the $s$-plane, one unstable pole of frequency \unit{4.43}{\rad~\second^{-1}} and a symmetric stable pole.  The objective is to reproduce the original system response rather than to just stabilise $G_0(s)$, so we choose $T_s = (2\pi)/(2 \times 30) \approx \unit{0.1}{\second}$.
Discretised using a Tustin transformation, the discretised baseline controller is:
\begin{equation}
K_0(z) =
\left[
\begin{array}{cc|cc}
0.6 & 0 & 3.2 & 0\\
0 & -0.2 & 0 & 25.6\\
\hline
-0.384 & -2.438 & 3.232 & 72
\end{array}
\right].
\end{equation}
Once more $D_K$ is nonzero (as would also be the case if a zero-order hold were to be used).  The baseline controller does \emph{not} tolerate a delay of $T_s$ added in the loop.  (Using the continuous time controller, this leads to oscillations.  In discrete time, it leads to instability.)
A transport delay of \unit{0.01}{\second} is tolerated though.
(The sampling period could, of course, be reduced so that a delay of a smaller $T_s$ is tolerated, however this can cause a longer prediction horizon (in time steps) to be needed to obtain stability when constraints are imposed, meaning a larger optimisation problem and heavier computational requirements.)

Assuming reduction of the sampling time is not acceptable, there are two options --- loop shifting (obtaining a predictor form observer), and dipole introduction (obtaining a filter form observer).  In this application, large external disturbances are not expected, and input constraints are not being considered.  Loop shifting will allow more time for computation --- $D_Ky(k)$ can be calculated very fast in comparison to solving the MPC QP, and the QP is formed using estimate $\hat{x}(k|k-1)$ so can commence at the previous time step.  (The loop shifting could even be implemented in continuous-time before discretisation).

There are 6 closed-loop poles in the original system ($n=4$, $n_K=2$).  The 2 (initially invisible) additional modes introduced in the observer are placed by designing a Kalman filter for system (\ref{eqn:freepolesystem}) with $Q=1$ and $R=10^7I$ --- i.e. assuming that most uncertainty comes from measurement noise.  The possible realisations are shown in Table~\ref{tbl:pendclpoles}.
\begin{table}[h]
\caption{Cart-pendulum closed-loop poles and relisations}
\label{tbl:pendclpoles}
\centering
\footnotesize
\renewcommand{\arraystretch}{1.25}
\begin{tabular}{cccc}
\hline\hline
& CPR1 & CPR2 & CPR3\\
\hline
$0.2416 + \mathrm{j}0.5304$ & S & O& S\\
$0.2416 - \mathrm{j}0.5304$ & S & O& S\\
$0.7832 + \mathrm{j}0.0630$ & S & S& O\\
$0.7832 - \mathrm{j}0.0630$ & S & S& O\\
$0.8800$ & O & S & S\\
$0.9708$ & O & S & S\\
\hline
New poles & $0.354\pm\mathrm{j}0.624$ & $0.242 \pm \mathrm{j}0.530$ & $0.515\pm \mathrm{j}0.763$\\
\hline
$\|G_{y\rightarrow \hat{y}}\|_{2}$ & $19.61$ &  $3.62$ & $6.59$\\
\hline
\hline
\end{tabular}
\end{table}
Realisation CPR2 is chosen, having lowest $\mathcal{H}_2$ gain from the output to the estimate of the output.

\subsubsection{Constrained MPC realisation}
Due to the loop-shifting, the state in the prediction model is augmented with the signal $D_Kr$, as this is required for correct prediction and enforcement of input constraints when the reference setpoint is non-zero, as well as the cost function (\ref{eqn:stagecost}) is modified to penalise $\|R^{1/2}(u - K_c(x-x_r))\|_2^2$.

Because $K_0(z)$ was in the forward path, pre-filtering is necessary when the observer is in the return path.  The form (\ref{eqn:bdkprefilter}) is used to ensure that the pendulum angle and angular velocity are not included in the state reference trajectory $x_r$.  The control gain $K_c$ is of size $1 \times 4$, so it is possible to force $3$ states in the filtered reference state to be equal to arbitrary values. The matrices
\begin{align}
L_1 & =
\begin{bmatrix}
0 & 1 & 0 & 0\\
0 & 0 & 1 & 0\\
0 & 0 & 0 & 1
\end{bmatrix}&
L_2 & =
\begin{bmatrix}
0 & 0\\
0 & 0\\
0 & 0\\
\end{bmatrix}
\label{eqn:prefiltoption}
\end{align}
constrain the cart reference velocity, and the angular state references to be zero.  All shaping is done using the cart position reference.  Figure~\ref{fig:prefilters} shows three prefilters.  Prefilter 1 is implemented using (\ref{eqn:prefilternominal1}).  Prefilter 2 is implemented using (\ref{eqn:bdkprefilter}) and the values of $L_1$ and $L_2$ above.

\begin{rmk}
\label{rmk:l1l2}
 An alternative would be to set
\begin{align}
L_1 & =
\begin{bmatrix}
1 & 0 & 0 & 0\\
0 & 0 & 1 & 0\\
0 & 0 & 0 & 1
\end{bmatrix}&
L_2 & =
\begin{bmatrix}
1 & 0\\
0 & 0\\
0 & 0\\
\end{bmatrix}
\end{align}
causing the position reference to be equal to its nominal setpoint value, and angular state references to be zero.  Reference shaping is performed using the cart velocity reference.
\end{rmk}

Prefilter 3 uses  $L_1$ and $L_2$, chosen as in Remark~\ref{rmk:l1l2}. Prefilter 2 is used for MPC simulations because any point of the filtered reference trajectory is a stable equilibrium.
\begin{figure}[htbp]
\centering
\includegraphics[width=0.475\textwidth]{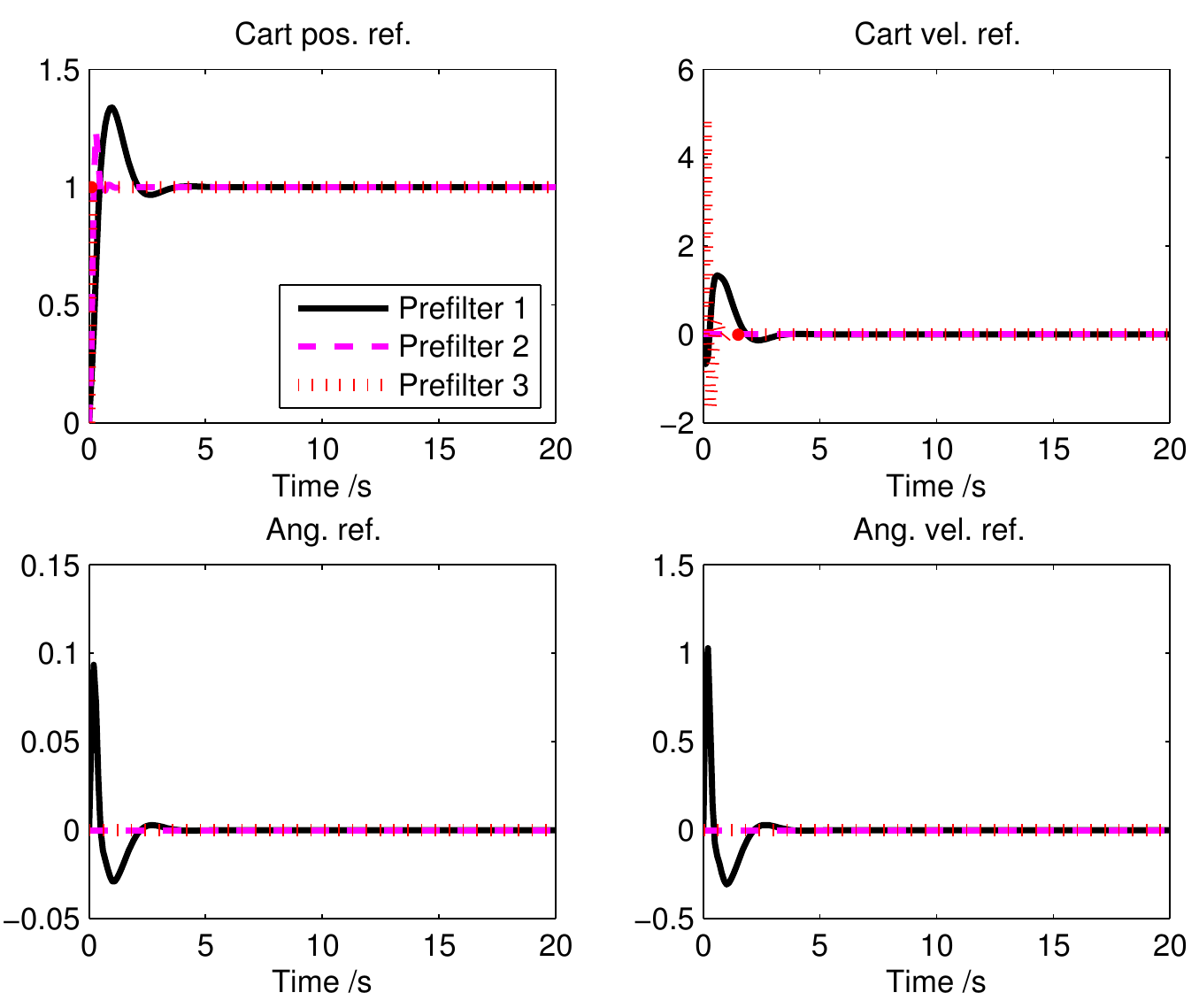}
\caption{Cart-pendulum prefilter realisations}
\label{fig:prefilters}
\end{figure}
\begin{table}[h]
\centering
\renewcommand{\arraystretch}{1.25}
\footnotesize
\caption{Cart-Pendulum simulations}
\begin{tabular}{ccccc}
\hline\hline
 & $K_0(s)$ & $K_0(z)$ & Case 1 & Case 2\\
\hline
$T_s$ & Continuous & \unit{0.1}{\second} & \unit{0.1}{\second} & \unit{0.1}{\second}\\
Realisation & -- & -- & CPR2 & CPR2\\
Prefilter & -- & -- & (\ref{eqn:bdkprefilter})/(\ref{eqn:prefiltoption}) & (\ref{eqn:bdkprefilter})/(\ref{eqn:prefiltoption})\\
%
%
%
%
%
%
Horizon & -- & -- & 15 & 15\\
Constraints & -- & -- & None &
$
\begin{bmatrix}
|\dot{x}|\\ |\theta| \\ |\dot{\theta}|
\end{bmatrix}
\le
\begin{bmatrix}
0.7\\ 0.175 \\ 0.3
\end{bmatrix}
$\\
\hline
\hline
\end{tabular}
\end{table}

\begin{figure*}[htbp]
\includegraphics[width=\textwidth]{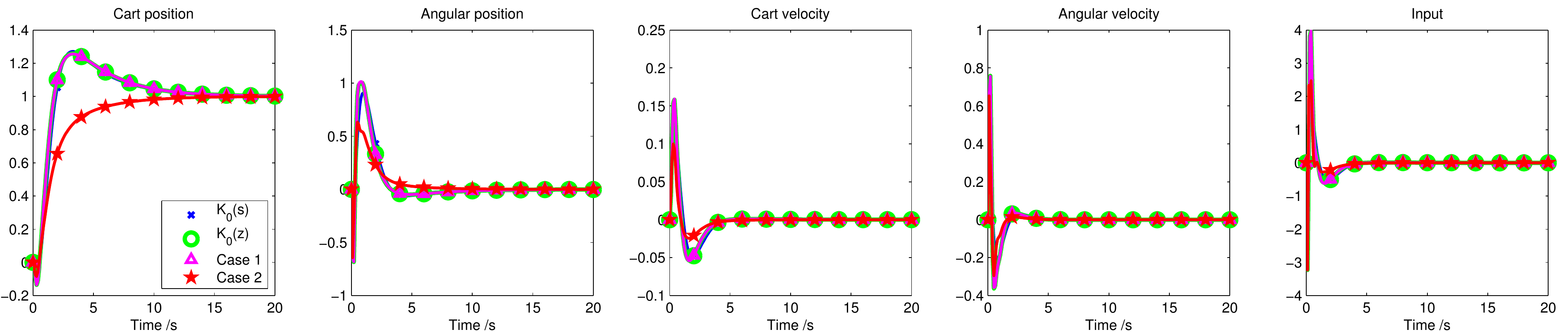}
\caption{Cart-Pendulum trajectories}
\label{fig:pendmatch}
\end{figure*}

\begin{rmk}
A predictor-form observer is used to estimate $\hat{x}(k+1|k)$, from which the MPC determines a control action to apply at time $k+1$.  At time $k=0$ the MPC cannot contribute to the control action.  Figure~\ref{fig:stairs} shows a zero-order hold plot of the pendulum angular velocity for the first \unit{2}{\second} of simulation.  The direct feedthrough due to loop-shifting acts at $t=0$, violating constraints at $t=0.1$.  However, the MPC controller quickly corrects this by $t=0.2$ and thereafter.  If a filter structure with a dipole had been used, this behaviour would not be exhibited.
\end{rmk}

\begin{figure}[htbp]
\centering
\includegraphics[width=0.3\textwidth]{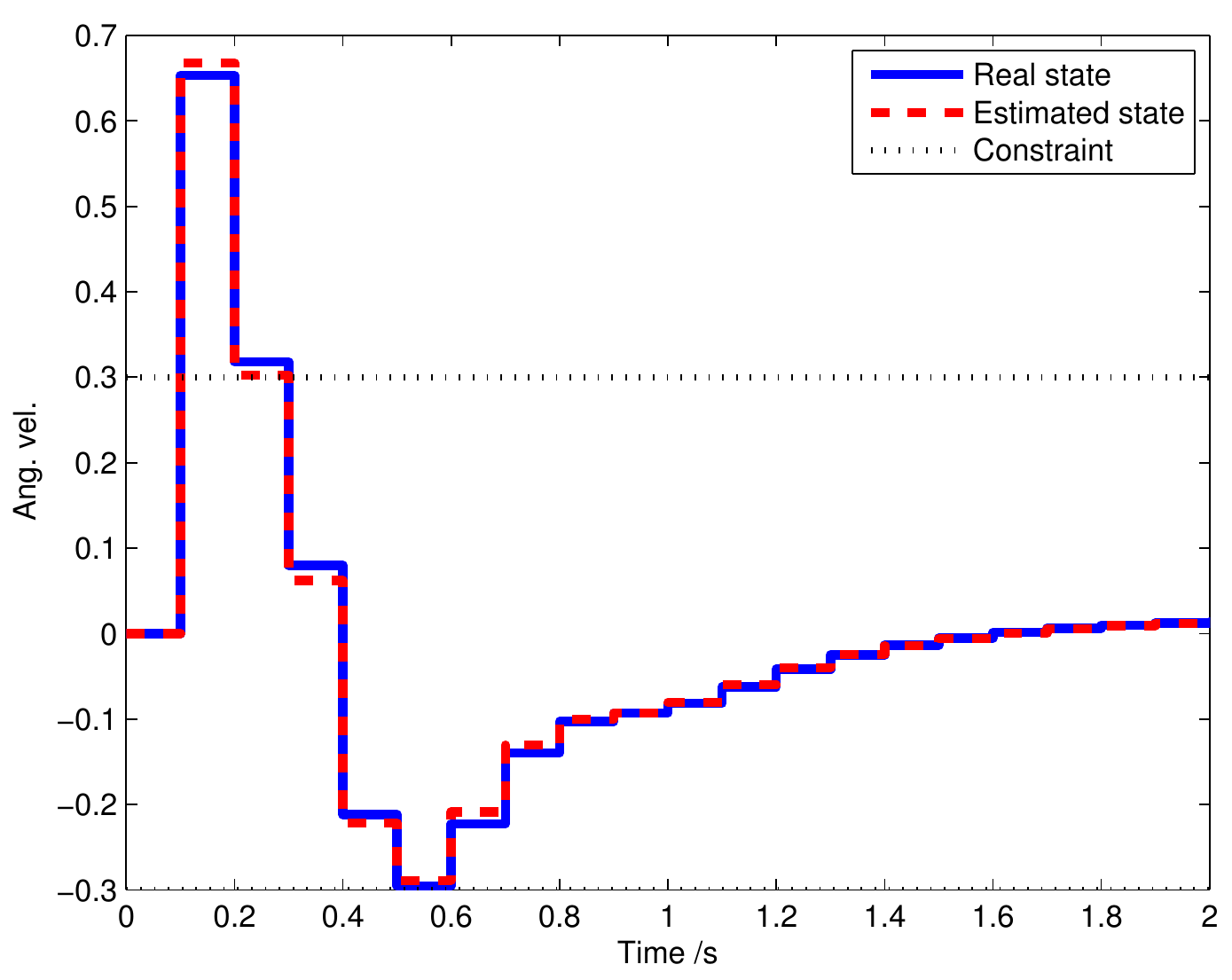}
\caption{Initial delay}
\label{fig:stairs}
\end{figure}

With a prediction horizon of $20$, $2$ input constraints and $6$ output constraints and $3$ slacks for constraint softening per time step, a dense QP with 60 decision variables and $90$ constraints is formed.  With the same setup as described previously, a solution time of $\unit{0.25}{\milli\second}$ is possible.  Linear scaling indicates that this could be implemented on very modest hardware indeed.

\subsection{Control of a Large Airliner --- MIMO control of a highly cross-coupled plant}
\label{sec:b747}
The third example extends \citep{JM2009} by using Proposition~\ref{prop:Ginf} as the heuristic for choosing a solution $T$ that yields a suitable observer realisation.  Furthermore, in the case of plant failure, the cost function (\ref{eqn:stattrack}) better encodes the contingency objectives than (\ref{eqn:stagecost}) whilst retaining fidelity in the nominal case.

\subsubsection{Model}
The plant model and baseline controller is provided by the simulator of a Boeing 747-100/200 that was used by the fault tolerant control group (AG16) of the Group for Aeronautical Research Europe (GARTEUR) \citep{Linden1996,MJ2003,JM2009,ELS2010,BS2011}.  Of interest for this demonstration, are 14 plant states (excluding lateral position) and 27 individually addressable control surfaces, consisting of 4 ailerons, 12 spoiler panels, 2 rudders, 4 elevators, 4 engines and a stabiliser.
\begin{table}[htbp]
\caption{States and Inputs}
\scriptsize
\centering
\renewcommand{\arraystretch}{1.25}
\hfill
\subfigure[States]{
\begin{tabular}{ccc}
\hline\hline
State & Symbol & Unit\\
\hline
Roll rate & $p$ & \unit{}{\rad\:\second^{-1}}\\
Pitch rate & $q$ & \unit{}{\rad\:\second^{-1}}\\
Yaw rate & $r$ & \unit{}{\rad\:\second^{-1}}\\
True airspeed & $V_{\mathrm{TAS}}$ & \unit{}{\meter\:\second^{-1}}\\
Angle of attack & $\alpha$ & \unit{}{\rad}\\
Sideslip & $\beta$ & \unit{}{\rad}\\
Roll & $\phi$ & \unit{}{\rad}\\
Pitch & $\theta$ & \unit{}{\rad}\\
Yaw & $\psi$ & \unit{}{\rad}\\
Height & $h$ & \unit{}{\meter}\\
Engine 1 & EPR1 & --\\
Engine 2 & EPR2 & --\\
Engine 3 & EPR3 & --\\
Engine 4 & EPR4 & --\\
\hline
\hline
\end{tabular}}
\hfill
\subfigure[Control surfaces]{
\begin{tabular}{cc}
\hline\hline
ID & Description\\
\hline
1 & Right inboard aileron\\
2 & Left inboard aileron\\
3 & Right outboard aileron\\
4 & Left outboard aileron\\
5--10 & Right spoiler panels\\
11--16 & Left spoiler panels\\
17 & Right inboard elevator\\
18 & Left inboard elevator\\
19 & Right outboard elevator\\
20 & Left outboard elevator\\
21 & Stabiliser\\
22 & Upper rudder\\
23 & Lower rudder\\
24--27 & Engines 1--4\\
\hline
\hline
\end{tabular}
}\hfill\hfill
\end{table}

The measured outputs are: roll rate ($p$), pitch rate ($q$), yaw rate ($r$), true airspeed ($V_{\mathrm{TAS}}$), roll angle ($\phi$), pitch angle ($\theta$), yaw angle ($\psi$), height ($h$) and rate of height change ($\dot{h}$).

The prediction model is obtained by averaging two empirically linearised models (linearised with opposite sign deflections) about a trimpoint obtained in continuous time using the Simulink \texttt{linmodv5} command to obtain an empirically linearised model about a trim point.  To further avoid the effects of nonlinearities, interaction between lateral and longitudinal states is explicitly nullified in the linearised model by setting the relevant elements of the state update matrix to zero.  The open-loop poles of the linearised continuous-time system range from $0$ to \unit{0.18}{\hertz}.
The baseline controller in \citep{BS2011} is a family of switched single-loop linear controllers with parameter-varying gains and rate limits and saturations.  To simplify the reverse engineering, the altitude select in series with pitch select, heading select in series with roll select, airspeed controller and yaw damper have been extracted, implemented as a block-diagonal MIMO controller and linearised about the operating point (i.e. with the scheduled gains locked to their values at the trim point).  

The plant and baseline controller are both discretised with a sampling period $T_s=\unit{0.1}{\second}$. The plant is discretised using a zero-order hold and the controller with a Tustin transformation.  The linear model dimensions are provided in Table~\ref{tbl:plantsizes}.

\begin{table}[htbp]
\caption{Plant sizes}
\label{tbl:plantsizes}
\footnotesize
\centering
\renewcommand{\arraystretch}{1.25}
\begin{tabular}{ccccc}
\hline\hline
& $K_0$ & Delayed $K_0$ & $G_0$ & Augmented $G_{0}$\\
\hline
States & 14 & 17 & 14 & 21\\
Inputs & 9 & 9 & 27 & 27\\
Outputs & 27 & 27 & 9 & 9\\
\hline\hline
\end{tabular}
\end{table}

\subsubsection{Reverse engineered controller and observer performance}
Whilst the original control system is implemented as four separate loops considering height, roll, yaw damping and air speed separately, these loops are not wholly independent, being coupled by the plant dynamics.

Figure~\ref{fig:xtermorder} shows the non-zero elements of the discretised state-update matrix, and also shows a realisation in a re-ordered form to better show the structure of the system.  The original control loops are also coupled by the inputs.
\begin{figure}[htbp]
\centering
\includegraphics[width=0.475\textwidth]{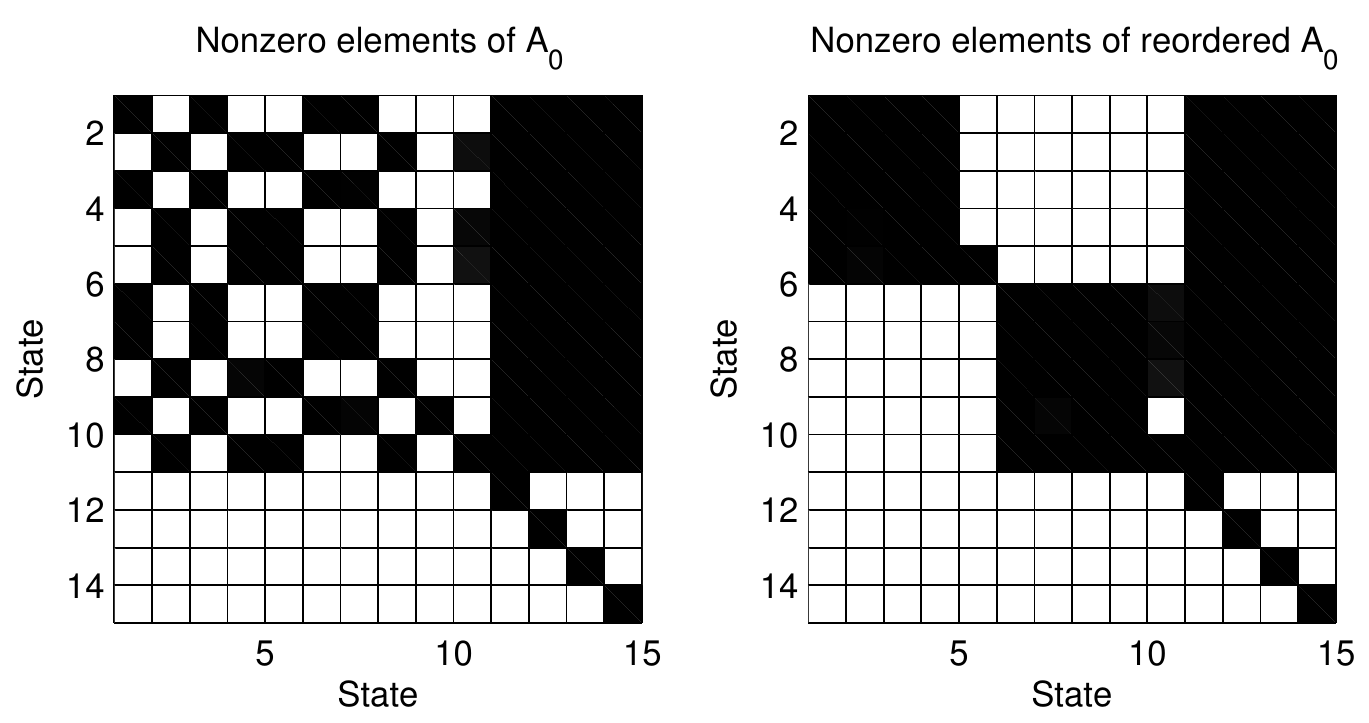}
\caption{Cross terms in linearised B747 model}
\label{fig:xtermorder}
\end{figure}

The baseline controller includes a number of integrators for offset-free tracking.  To cast this into an MPC framework, the linearised model is augmented with $7$ disturbance states, modelling constant disturbances added to the first derivatives of $p$, $q$, $r$, $V_{\mathrm{TAS}}$, angle of attack ($\alpha$), $\theta$ and $\psi$.
This is as many as can be tolerated whilst maintaining observability of pair $(C,A)$, and the chosen combination has been chosen experimentally to give the best observer performance.
The baseline controller is not strictly proper ($K_0(0) \ne 0$), and given the number of states and inputs, and therefore the relatively high complexity of solution of the MPC problem, adding a unit delay is the chosen method for obtaining a strictly proper realisation.

There are up to $^{38}C_{17}$ possible realisations.  However, the $7$ uncontrollable disturbance states must stay in the state-feedback dynamics, reducing this to $^{31}C_{17} \approx 2.7 \times 10^{8}$.  Further reductions in this number can be made by noting that $9$ of the closed-loop poles are complex conjugate pairs, reducing the number of combinations to $431415$.  A further reduction can be obtained by noting that the system
\begin{equation}
\left[
\begin{array}{cc|c}
A & B_d & BC_K\\
0 & I & 0\\
\hline
B_KC & 0 & 0
\end{array}
\right]
\end{equation}
has a further three uncontrollable modes in addition to the disturbances, corresponding to three poles at $z=0.9512$ (so these must remain in the state feedback dynamics), reducing the number of combinations to a somewhat more tractable total of $41958$.  The search using the metric in Proposition~\ref{prop:Ginf} can be carried out in \unit{181}{\second} using a script written using standard MATLAB Control Toolbox commands for analysis of LTI systems on a modestly specified desktop workstation.  Of the possible realisations searched, $8143$ have feasible solutions of $T$.

To demonstrate how critical the ``correct'' choice of realisation, Table~\ref{fig:b747combos} shows three possible combinations.  The first considers the first realisation found in the search (i.e. an arbitrary choice).  The second considers the minimisation of the sum of the absolute values of the observer poles (in theory, the fastest observer).  The third considers the metric of Proposition~\ref{prop:Ginf}.

\begin{table}[htbp]
\caption{Realisations of B747 controller}
\label{fig:b747combos}
\centering
\renewcommand{\arraystretch}{1.25}
\footnotesize
\begin{tabular}{cccc}
\hline\hline
& Realisation 1 & Realisation 2 & Realisation 3\\
\hline
How found? & First found & Min sum abs poles & Proposition~\ref{prop:Ginf}\\
Sum Abs Poles & $10.7585$ & $10.6207$ & $13.2767$ \\
$\|G_{y\rightarrow\hat{e}}\|_2$ & $1116$ & $3925$ & $67.5$\\
$\|G_{d\rightarrow\hat{x}}\|_2$ & $692.6$ & $2875$ & $598.5$\\
Proposition~\ref{prop:Ginf}  & $7.7\times10^{5}$ & $1.1 \times 10^7$ & $4.0\times 10^4$\\
\hline
\hline
\end{tabular}
\end{table}

Figure~\ref{fig:b747poles} graphically presents the open-loop pole distribution, and the closed-loop observer and state feedback pole allocations graphically for each realisation.
\begin{figure*}[htbp]
\centering
\includegraphics[width=0.85\textwidth]{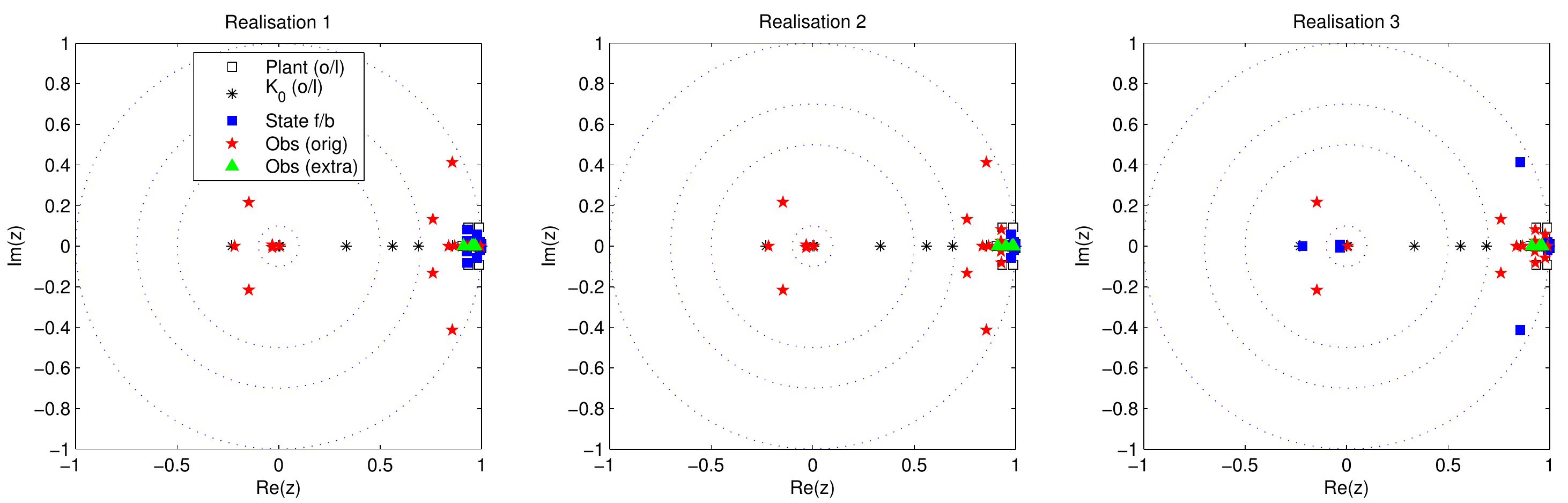}
\caption{Observer-based realisations for Boeing 747 controller --- Pole maps}
\label{fig:b747poles}
\end{figure*}
Figure~\ref{fig:compareobsturb} shows the observer performance for estimation of sideslip ($\beta$, unmeasured), roll angle (measured), yaw angle (measured) and engine power for one engine (unmeasured) in response to unmeasured turbulence, and Gaussian measurement noise, with the standard deviations shown in Table~\ref{tbl:measnoise}, for \unit{60}{\second} of nominally ``straight-and-level'' flight, under the control of the discretised baseline controller.  It can be seen that Realisation 2 (supposedly maximising the ``speed'' of the observer) is actually worse than the arbitrary choice in Realisation 1.  However, with both of these realisations, the measurement noise and disturbances are amplified to the the point that the state estimate is meaningless, rendering the MPC controller no more useful than a rather baroque realisation of $K_0(z)$.  On the other hand, Realisation 3 provides a very good filtered estimate of the measured states, and clear tracking of the salient features of the unmeasured state trajectories.
\begin{table}[htbp]
\caption{Measurement noise $3\sigma$ values (units as per states)}
\label{tbl:measnoise}
\centering
\footnotesize
\renewcommand{\arraystretch}{1.25}
\begin{tabular}{ccccccccc}
\hline
\hline
$p$ & $q$ & $r$ & $V_{\mathrm{TAS}}$ & $\phi$ & $\theta$ & $\psi$ & $h$ & $\dot{h}$\\
\hline
$10^{-3}$ & $10^{-3}$ & $10^{-3}$ & $10^{-2}$ & $10^{-3}$ & $10^{-3}$ & $10^{-2}$ & $10^{-1}$ & $10^{-3}$\\
\hline
\hline
\end{tabular}
\end{table}

\begin{figure*}
\hfill
\subfigure[Realisation 1]{\includegraphics[width=0.3\textwidth]{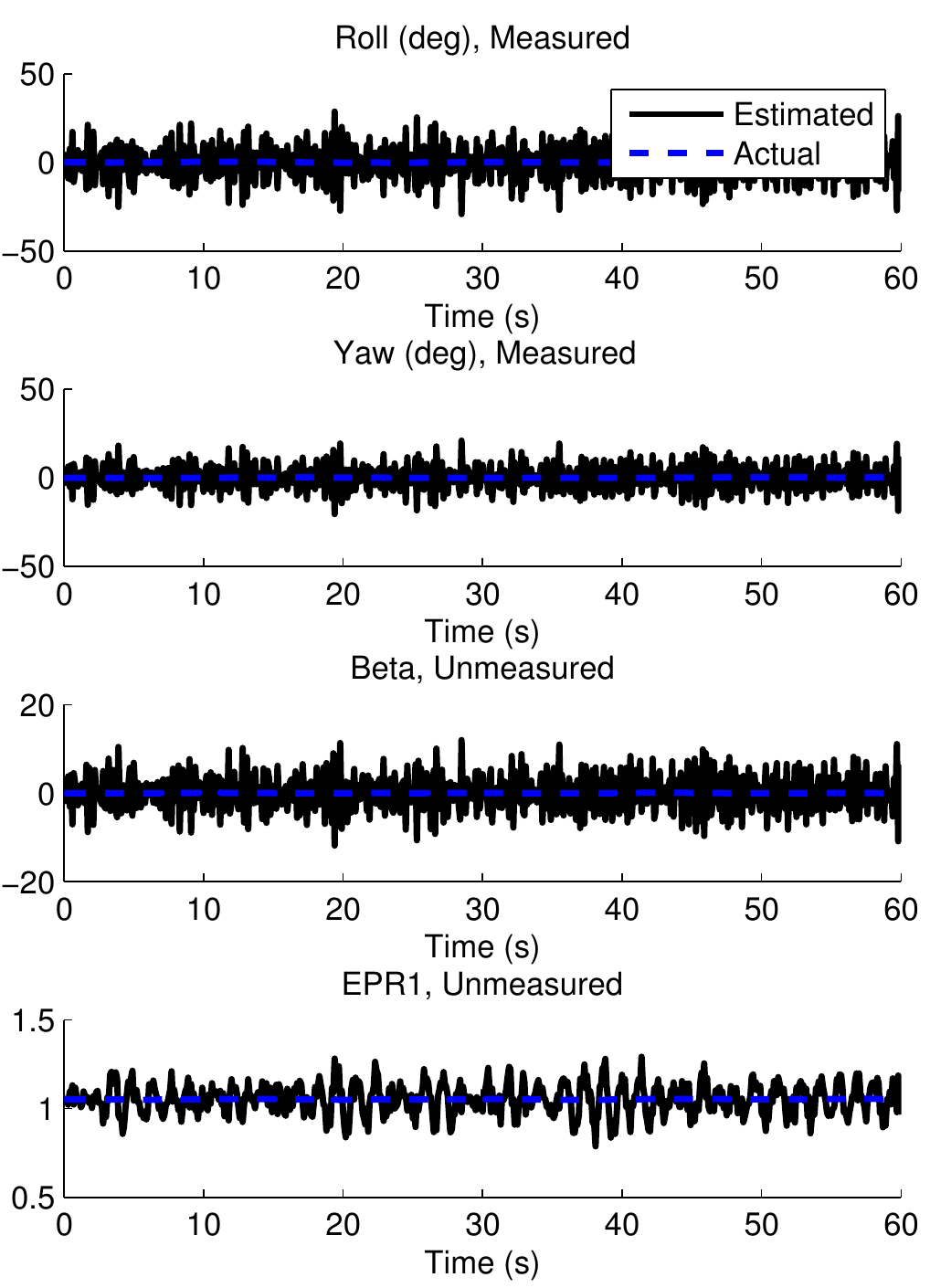}}
\hfill
\subfigure[Realisation 2]{\includegraphics[width=0.3\textwidth]{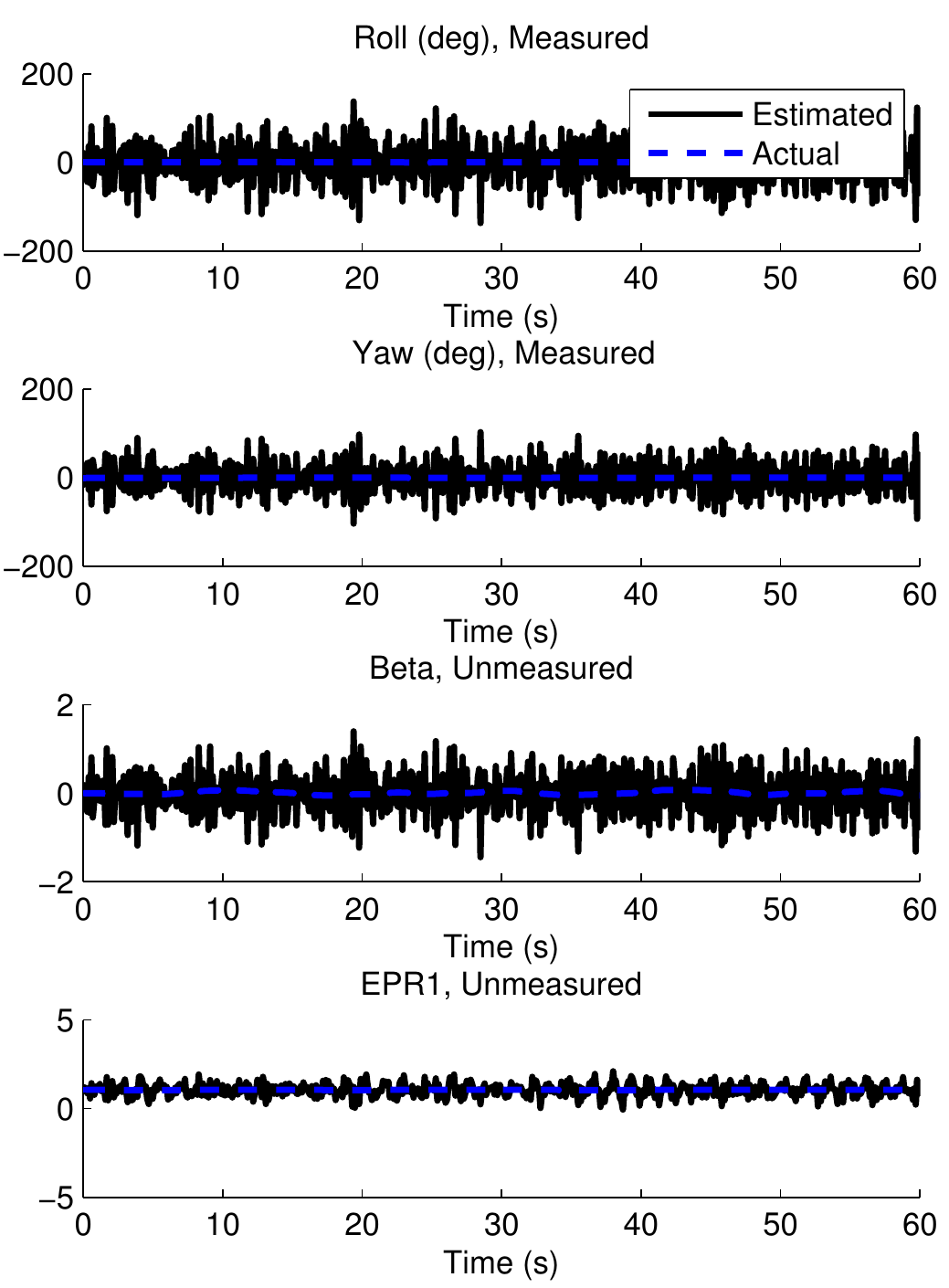}}
\hfill
\subfigure[Realisation 3]{\includegraphics[width=0.3\textwidth]{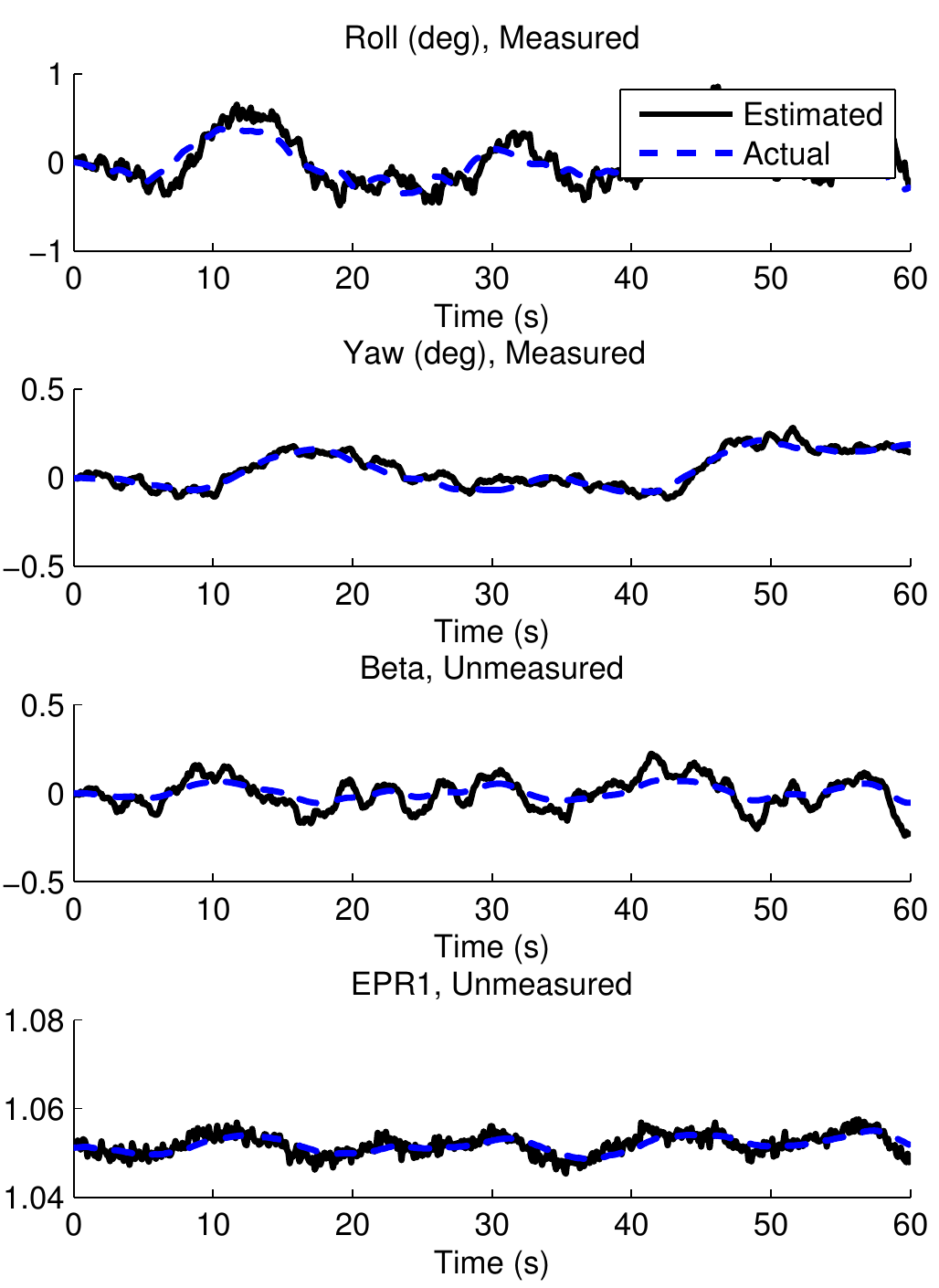}}
\hfill\hfill
\caption{Observer performance with different realisations (straight and level flight with turbulence and sensor noise)}
\label{fig:compareobsturb}
\end{figure*}

\subsubsection{Controller performance}
Unconstrained matching in the tracking a piecewise linear trajectory of yaw, height and true airspeed setpoints using cost function (\ref{eqn:stattrack}) and Realisation 3 is shown in Figure~\ref{fig:mpctrack}.  It is of no surprise that the reverse-engineered controller matches the behaviour of the simplified baseline controller --- the random number seeds used for sensor noise and turbulence are identical for each simulation, and the controllers should behave identically.  Figure~\ref{fig:mpctrack} also shows the closed-loop trajectory when all ailerons, spoilers $1$ and $4$--$8$, left inboard and right outboard elevators are locked at their trim positions.  The baseline controller is unable to perform the yaw man{\oe}uvre, although even this is robust enough to track height and $V_{\mathrm{TAS}}$ well, due to integral action.  However, it can be seen that despite the removal of these degrees of control freedom, that the reverse-engineered controller, given knowledge of the constraints and a prediction horizon $N=10$, still has access to sufficient control authority to approximately track the required trajectory, although there is some loss of performance.

If cost function (\ref{eqn:stagecost}) were to have been used, the trajectory would have been very oscillatory and failed to track the trajectory.  This is clear from the interpretation of the cost function.  When faced with control surfaces constrained to zero, (\ref{eqn:stagecost}) will attempt to regulate the plant state to a subspace where the relevant elements of $u=K_c\hat{x}$ are zero.  For example, when performing a change in heading, it is natural for the aircraft to perform a roll, but cost (\ref{eqn:stagecost}) will try to counteract this due to the locked aileron.  On the other hand, (\ref{eqn:stattrack}) will use other degrees of freedom (i.e. the spoiler panels) to as best as possible effect the ``intention'' of $u=K_c\hat{x}$ even if the realisation is rather different.

\begin{figure}
\centering
\includegraphics[width=0.35\textwidth]{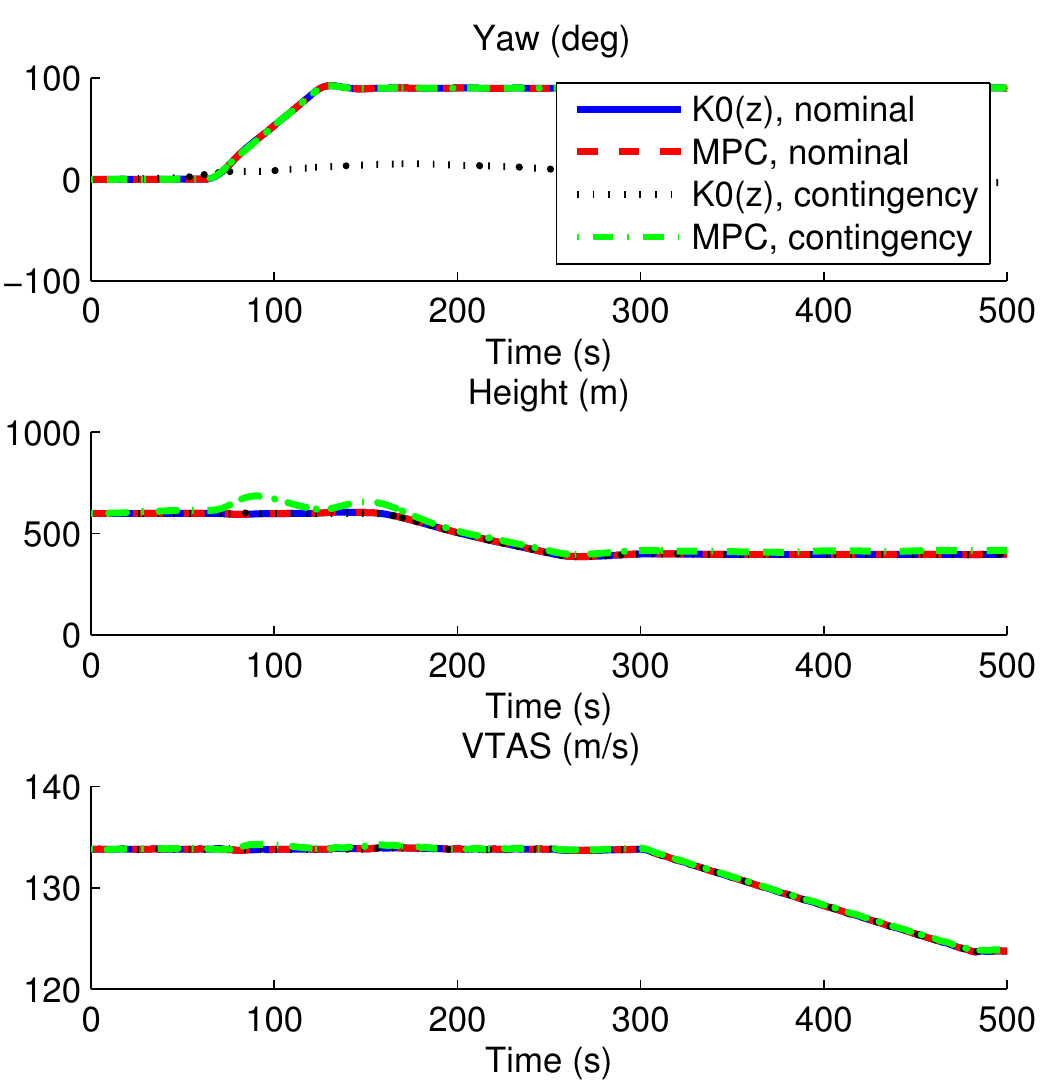}
\caption{Tracking trajectories}
\label{fig:mpctrack}
\end{figure}

\section{Conclusions}
A method has been presented for obtaining a constrained MPC controller that behaves in the same way as an existing LTI output-feedback controller when constraints are not active.  The method accounts for offset-free output tracking and can be implemented using standard tools.  A new heuristic has been presented for choosing the non-unique observer-based realisation upon which the MPC controller is based.  Three examples of different complexity are presented to demonstrate the efficacy and usefulness of the scheme.  Two of these apply to an approximately linear operating region of a non-linear plant.  The resulting MPC controllers are able to enforce constraints, and use disturbance estimation and plant redundancy to provide a level of reconfigurability and robustness to plant failures, as would befit a constrained predictive controller designed from scratch.

\section*{Acknowledgements\addcontentsline{toc}{section}{Acknowledgements}}
This work was funded by EPSRC grant EP/G030308/1, the European Space Agency and EADS Astrium.

\bibliographystyle{hyperlink_bib}
\bibliography{mpcrefs,fpgarefs}

\end{document}